\documentclass{amsart}
\usepackage[utf8]{inputenc}
\usepackage{amssymb,bm,amsfonts, stmaryrd, amsmath, amsthm, enumerate, mathtools,comment}
\usepackage[all]{xy}
\makeatletter
\@namedef{subjclassname@2020}{\textup{2020} Mathematics Subject Classification}
\makeatother
\usepackage[colorlinks]{hyperref}
\hypersetup{
 colorlinks=true,
 linkcolor=blue,
 filecolor=magenta, 
 urlcolor=cyan,
}
\usepackage{wrapfig,tikz, tikz-cd}
\usetikzlibrary{arrows}
\usepackage[capitalise]{cleveref}
\crefformat{equation}{(#2#1#3)}
\crefrangeformat{equation}{(#3#1#4--#5#2#6)}
\crefformat{enumi}{(#2#1#3)}
\crefrangeformat{enumi}{(#3#1#4--#5#2#6)}
\newtheorem{theorem}{Theorem}[section]
\newtheorem{lemma}[theorem]{Lemma}
\newtheorem{proposition}[theorem]{Proposition}
\newtheorem{corollary}[theorem]{Corollary}
\newcounter{intro}
\newtheorem{questionx}{Question}
\newtheorem{introthm}[intro]{Theorem}

\theoremstyle{definition}

\newtheorem{example}[theorem]{Example}
\newtheorem{remark}[theorem]{Remark}

\newtheorem{construction}[theorem]{Construction}

\newtheorem{chunk}[theorem]{}

\newtheorem{thm}{Theorem}[subsection]

\theoremstyle{definition}
\newtheorem{ex}[thm]{Example}
\newtheorem{rem}[thm]{Remark}

\newtheorem{ch}[thm]{}
\newtheorem{notation}[thm]{Notation}

\newcommand{\ges}{\geqslant}
\newcommand{\les}{\leqslant}

\newcommand{\Spec}{{\operatorname{Spec}}}

\newcommand{\spec}{{\operatorname{Spec}^*}}

\newcommand{\Hom}{{\operatorname{Hom}}}

\newcommand{\Ext}{{\operatorname{Ext}}}
\newcommand{\Tor}{{\operatorname{Tor}}}
\newcommand{\RHom}{\operatorname{\mathsf{RHom}}}
\newcommand{\del}{\partial}
\renewcommand{\H}{\operatorname{H}}
\renewcommand{\S}{\mathcal{S}}
\DeclareMathOperator{\thick}{\mathsf{thick}}
\DeclareMathOperator{\codim}{codim}
\DeclareMathOperator{\codepth}{codepth}
\DeclareMathOperator{\coker}{coker}
\DeclareMathOperator{\depth}{depth}
\DeclareMathOperator{\ad}{ad}
\newcommand{\supp}{{\operatorname{Supp}}}
\newcommand{\xra}{\xrightarrow}

\newcommand{\vp}{\varphi}
\newcommand{\f}{\bm{f}}
\newcommand{\V}{{\rm{V}}}

\newcommand{\cV}{{\mathcal{V}}}
\renewcommand{\ll}{\ell\ell}

\newcommand{\ann}{{\operatorname{ann}}}
\newcommand{\cat}{{\mathsf{LScat}}}
\newcommand{\shift}{{\scriptstyle\mathsf{\Sigma}}}
\newcommand{\lotimes}{\otimes^{\sf L}}
\DeclareMathOperator{\D}{\mathsf{D}}

\DeclareMathOperator{\height}{height}
\DeclareMathOperator{\cid}{cid}
\DeclareMathOperator{\rank}{rank}
\DeclareMathOperator{\rad}{rad}
\DeclareMathOperator{\Df}{\mathsf{{D}^{f}}}
\newcommand{\T}{\mathsf{T}}
\newcommand{\level}{\mathsf{level}}
\newcommand{\m}{\mathfrak{m}}
\newcommand{\p}{\mathfrak{p}}

\renewcommand{\t}{\mathsf{t}}
\newcommand{\h}{\mathsf{h}}

\title[Bounds on cohomological support varieties]{Bounds on cohomological support varieties}
\author[B.~Briggs]{Benjamin Briggs}
\address{Department of Mathematical Sciences,
University of Copenhagen,
Universitetsparken 5
DK-2100 Copenhagen \O}
\email{bpb@math.ku.dk}
\author[E.~Grifo]{Elo\'{i}sa Grifo}
\address{Department of Mathematics,
University of Nebraska, Lincoln, NE 68588, U.S.A.}
\email{grifo@unl.edu}
\author[J.~Pollitz]{Josh Pollitz}
\address{Mathematics Department, 
Syracuse University, 
Syracuse, NY 13244 U.S.A.}
\email{jhpollit@syr.edu}


\begin{document}

\begin{abstract}
Over a local ring $R$, the theory of cohomological support varieties attaches to any bounded complex $M$ of finitely generated $R$-modules an algebraic variety $\V_R(M)$ that encodes homological properties of $M$. We give lower bounds for the dimension of   $ \V_R(M)$ in terms of classical invariants of $R$. 
In particular, when $R$ is Cohen-Macaulay and not complete intersection we find that there are always varieties that cannot be realized as the cohomological support of any complex. When $M$ has finite projective dimension, we also give an upper bound for $ \dim \V_R(M)$ in terms of the dimension of the radical of the homotopy Lie algebra of $R$. This leads to an improvement of a bound due to Avramov, Buchweitz, Iyengar, and Miller on the Loewy lengths of finite free complexes, 
and it recovers a result of Avramov and Halperin on the homotopy Lie algebra of $R$.
Finally, we completely classify the varieties that can occur as the cohomological support of a complex over a Golod ring.
\end{abstract}

\maketitle

\vspace{-.2in}
\section*{Introduction}\label{s_intro}

In local algebra, the theory of cohomological support varieties attaches to any complex $M$ of modules over a local ring $R$ an algebraic variety $\V_R(M)$, defined over the residue field of $R$. These were first introduced by Avramov~\cite{Avramov:1989} for local complete intersection rings, taking inspiration from modular representation theory~\cite{Quillen:1971}. They were used by Avramov and Buchweitz to establish striking symmetry results in the homological behavior of modules over complete intersection rings \cite{Avramov/Buchweitz:2000b}. The theory has been extended in  \cite{Avramov/Buchweitz:2000b,Avramov/Iyengar:2018,Burke/Walker:2015,Jorgensen:2002,Pollitz:2019,Pollitz:2021}, and now encompasses all local rings. 
These varieties encode important homological information about $M$, as well as ring theoretic properties of~$R$. For example,  the third author showed \cite{Pollitz:2019} that $R$ is  complete intersection if and only if $\V_R(R)$ is a point, and thus nontrivial varieties act as obstructions to the complete intersection property.

Throughout, $R$ will be a noetherian local ring with residue field $k$. By Cohen's structure theorem, the completion of $R$ at its maximal ideal has the form $Q/I$
where $(Q,\m)$ is a regular local ring and $I$  is an ideal minimally generated by $n$ elements $f_1, \ldots, f_n$ of $\m^2$. The cohomological support variety of a complex $M$ of $R$-modules is a conical, algebraic subvariety
\[
\V_R(M)\subseteq \mathbb{A}_k^{n}
\]
(equivalently, a projective subvariety of $\mathbb{P}_k^{n-1}$ of dimension one less). 
In this paper, we give both upper and lower bounds for $\dim \V_R(M)$, in terms of other invariants of $R$ and $M$, for any bounded complex $M$ of finitely generated $R$-modules.

As a consequence, we give a negative answer to the \emph{realizability question}: 

\begin{questionx}\label{q_realisation}
Given a fixed $R$,  can every conical subvariety of $\mathbb{A}_k^n$ be realized as the cohomological support of some bounded complex of finitely generated $R$-modules?
\end{questionx}

When $R$ is complete intersection, Bergh \cite{Bergh} showed that any conical variety is the cohomological support of some $R$-module; see also \cite{Avramov/Iyengar:2007}. In general, the third author established a partial negative answer: \emph{if $R$ is  not  complete intersection then $\{0\}$ is not the cohomological support of any finite $R$-module}~\cite[Theorem B]{Pollitz:2021}; but the proof there says nothing about the cohomological supports of  \emph{complexes}.

One motivation for \cref{q_realisation} is that constraints on the cohomological support impose constraints on \emph{the lattice of thick subcategories of the bounded derived category $R$} (see \cref{s_level}). In particular, if $\{0\}$ cannot be realized as the  support of any complex, then cohomological supports can be used to detect when thick subcategories are zero.  For complete intersection rings, the thick subcategories were classified by Stevenson via cohomological support (phrased in triangulated terms) \cite{Stevenson:2014a}, and the lattice of all such categories has good geometric properties. The situation for general local rings seems to be markedly different; see \cref{cor_thick_intersection} for the case of Cohen-Macaulay rings, and \cref{cor Lzeta intersection golod} for the case of Golod rings.

Our non-realizability result follows from a lower bound on the dimension of cohomological support varieties in terms of the following invariants of $R$: the deviations $\varepsilon_1(R)=\dim Q$ and $\varepsilon_2(R)=n$, as well as $\depth(R)$ and, if $R$ is artinian, the Loewy length $\ll_R(R)$.

\begin{introthm}\label{theorem A}
For any local ring $R$ and any bounded complex $M$ of finitely generated modules with $\H(M)\neq 0$, we have
\[
\dim \V_R(M) \ges \varepsilon_2(R)- \varepsilon_1(R) + \depth(R)\,,
\]
with strict inequality if $R$ is not complete intersection. Moreover, if $R$ is artinian then 
\[
\dim \V_R(M) \ges \varepsilon_2(R)-\ll_R (R)+1\,.
\]
\end{introthm}

Both inequalities in the theorem follow from a stronger bound in terms of the \emph{Lusternik-Schnirelmann category} of $R$; see \cref{cor_bound}. 

For Cohen-Macaulay rings, the lower bound  $\varepsilon_2(R)- \varepsilon_1(R) + \depth(R)$ is known as the {\bf complete intersection defect} of $R$. This integer is always nonnegative, and vanishes exactly when $R$ is complete intersection \cite{KiehlKunz,Avramov:1977}. Thus \cref{theorem A} yields a strong negative answer to the realizability question: \emph{if $R$ is Cohen-Macaulay but not complete intersection then}
\[
\dim \V_R(M)\ges 2
\]
\emph{for any bounded complex $M$ of finitely generated modules with $\H(M)\neq 0$.} 


Our upper bounds on the dimension of $\V_R(M)$ exploit a new connection between the cohomological support varieties and the \emph{homotopy Lie algebra} of $R$. This is a graded Lie algebra $\pi^*(R)$ attached to $R$ that has seen a number of important applications in local algebra (cf.~\cite[Section 10]{Avramov:2010}); the $k$-vector space ranks of its graded pieces are the deviations of $R$. The next result bounds $\dim \V_R(M)$ in terms of the subspace $\rad(\pi^*(R))$ of radical elements, first studied in \cite{Felix/Halperin/Jacobsson/Lofwall/Thomas:1988}. See \cref{s_HLA_background} for definitions and discussion.

\begin{introthm}\label{theorem C}
For any local ring $R$ and any bounded complex $M$ of finitely generated free modules
\[
\dim \V_R(M) \leqslant \varepsilon_2(R)- \rank_k \rad^2(\pi^*(R))\,.
\]
\end{introthm}

This is really a result about $\V_R(R)$ since the complexes considered satisfy $\V_R(M)=\V_R(R)$. \Cref{theorem C} follows from the more precise \Cref{th_radical_support}, which describes how each radical element of $\pi^*(R)$ gives rise to a hyperplane containing $\V_R(R)$. In particular, we obtain an upper bound on the rank of the smallest vector subspace of $\mathbb{A}^n_k$ containing $\V_R(M)$, which is often larger than $\dim \V_R(M)$. 

Combined with the characterization  \cite{Pollitz:2019} of local complete intersection rings in terms of $\V_R(R)$, \cref{theorem C} immediately implies a result of Avramov and Halperin \cite{Avramov/Halperin:1987}, that if every element of $\pi^2(R)$ is radical then $R$ is complete intersection; see \cref{cor_AH_rad}. A related open problem of Avramov \cite[Problem 4.3]{Avramov:1989a} connects the existence of embedded deformations to central elements in $\pi^2(R)$. In \cref{remark embedded deformation} we discuss how our results shed some new light on this question.

As another consequence of Theorem \ref{theorem C}, as well as the proof of \cref{theorem A}, is the following improvement to a central result from \cite{Avramov/Buchweitz/Iyengar/Miller:2010} that bounds the Loewy length of finite free complexes in terms of the \emph{conormal free rank} $\mathrm{cf\text{-}rank}(R)$.

\begin{introthm}\label{intro upper bound on dim rad}
For any local ring $R$ and any bounded complex $M$ of finitely generated free modules with $\H(M)\neq 0$, we have
\[
\sum \ell\ell_R \H_n(M) \ges \varepsilon_2(R)- \dim \V_R(R)+1 \ges \mathrm{cf\text{-}rank}(R) +1\,.
\]
\end{introthm}
 
The inequality between the outer two terms is \cite[Theorem 10.1]{Avramov/Buchweitz/Iyengar/Miller:2010}. Equality can hold throughout,  for example if $R$ is complete intersection and $M$ is the Koszul complex on a minimal set of generators for the maximal ideal of $R$. The inequalities in the theorem above are established using \Cref{upper bound on dim rad,r_cf_rank}, and the right hand inequality is often strict; therefore our result above provides a stronger lower bound on $\sum \ell\ell_R \H_n(M)$ than the previous one from \cite{Avramov/Buchweitz/Iyengar/Miller:2010}. More specifically, a consequence of \Cref{th_radical_support} is that 
\[
\varepsilon_2(R)- \rank_k (\mathrm{span}_k \V_R(R))+1 \ges \mathrm{cf\text{-}rank}(R) +1\,,
\]
and hence when $\V_R(R)$ is not a linear space, \cref{intro upper bound on dim rad} gives a better lower bound than the one in \cite{Avramov/Buchweitz/Iyengar/Miller:2010}.
Taking $R$ to be the group algebra of an elementary abelian $p$-group also recovers a theorem from homotopy theory \cite{Carlsson:1983,Allday/Puppe:1993} (see \cref{r_cf_rank}), although in this case our bound coincides with that of \cite{Avramov/Buchweitz/Iyengar/Miller:2010}.

The Golod property is reflected in the cohomological support varieties, and for these rings we can completely solve the realizability problem stated in \cref{q_realisation}. The following is \Cref{golod}:

\begin{introthm}
Let $R$ be a Golod ring. For any bounded complex $M$ of finitely generated  modules with $\H(M)\neq 0$, the cohomological support variety $\V_R(M)$ is either all of $\mathbb{A}_k^n$ or a (conical) hypersurface, and every hypersurface is indeed a cohomological support variety of some complex.
Moreover, if $R$ is a non-hypersurface ring, then $R$ has full support $\V_R(R)=\mathbb{A}_k^n$.  
\end{introthm}

This theorem also recovers the fact that if $R$ is Golod but not a hypersurface, then the conormal free rank of $R$ is zero; see \Cref{remark Golod no embedded defs} for details.

Section \ref{s_dg} contains the necessary background on cohomological support varieties, semifree dg algebras, levels, and the homotopy Lie algebra. We prove our lower bound on the dimension of cohomological support varieties in Section \ref{s_bounds}, which leads to our nonrealizability result. In Section \ref{s_HLA} we establish a relationship between the homotopy Lie algebra of $R$ and $\V_R(R)$, and give some applications. We also discuss how this appears to point to a deeper connection between the ring-theoretic properties of $R$, the geometry of $\V_R(R)$, and the structure of the homotopy Lie algebra. Finally, in Section \ref{s_golod} we focus on the case of Golod rings.

\section{Notation and background}\label{s_dg}

\subsection{Fixed notation}\label{notation_s_support}

We begin by introducing a few objects whose notation will remain fixed for the remainder of the paper. Throughout, $R$ is a local ring with chosen minimal Cohen presentation $\widehat{R}=Q/I$. That is, $(Q,\m,k)$ is a regular local ring and $I\subseteq \m^2$. We fix a  list of minimal generators $\f=f_1,\ldots,f_n$ for $I$; the number $n$ of generators will be important throughout.

The Koszul complex on $\f$ over $Q$ will be denoted 
\[
E \coloneqq Q[ e_1,\ldots,e_n\mid \del e_i=f_i]\,.
\]
As an algebra, $E$ is the exterior algebra over $Q$ on  variables $e_1,\ldots,e_n$ in bijection with $f_1,\ldots,f_n$, given homological degree $1$. As a differential graded (henceforth \emph{dg}) algebra, $E$ is also given the differential uniquely determined by $\del( e_i)\coloneqq f_i$.

We will also write
\[
\S\coloneqq k[\chi_1,\ldots,\chi_n]\,,
\]
the polynomial algebra over $k$ on variables $\chi_1,\ldots,\chi_n$ dual to $e_1,\ldots,e_n$, with each $\chi_i$ given cohomological degree $2$ (or homological degree $-2$).

\subsection{Graded support}
Throughout, let $\spec \S$ denote the set of homogeneous primes $\p$ in $\S$ equipped with the Zariski topology, with closed subsets of the form $
\cV(\mathcal{I})\coloneqq\{\p\in \spec \S:\mathcal{I}\subseteq \p\}
$ 
for a homogeneous ideal $\mathcal{I}$ of $\S$. We think of $\Spec^* \S$ as ``homogeneous affine space''; its points correspond to irreducible conical subvarieties of $\mathbb{A}^n_{\overline{k}}$. We denote by $\bm{0}$ the irrelevant maximal ideal $(\chi_1,\ldots,\chi_n)$ of $\S$, thought of as the origin of $\spec \S$.
\label{c_hsupport}

The support of a graded $\S$-module $X$ is 
\[
\supp_\S X\coloneqq \{\p\in \spec \S: X_\p\neq 0\}\,.
\]
When $X$ is finitely generated over $\S$, then $\supp_\S X$ is a closed subset of $\spec \S$: 
\[
\supp_\S X=\cV(\ann_\S X)\,.
\]
Any such support set is {conical} (the closure of a set of lines through the origin), and thus, if it is not zero, may be thought of as a projective variety embedded in $\mathbb{P}^{n-1}_k=\mathrm{Proj}\,\S$, with dimension one less. We work with affine varieties in order to exploit the functoriality of $\Spec^*\S$ with respect to local homomorphisms (as was done in \cite{Briggs/Grifo/Pollitz:2022}), and as well because it is sometimes useful to view $\supp_\S X$ as corresponding to a subset of a \emph{vector space} over $k$. 

We will say a closed set $V \subseteq \spec\S$ is {\bf linear} if $V=\cV(\mathcal{I})$ where $\mathcal{I}$ is an ideal generated by linear forms, that is, elements of $\S^2$.

\subsection{Dg algebras and dg modules}
Let $A=\{A_i\}_{i\in \mathbb{Z}}$ be a dg algebra. We write $\D(A)$ for its derived category of dg $A$-modules. This is regarded as a triangulated category in the standard way, with $\shift$ denoting the suspension functor.  We write $\Df(A)$ for the full subcategory of $\D(A)$
 consisting of those objects $M$ such that $\H(M) $ is finitely generated over $\H(A)$.  The reader is directed to \cite{Avramov/Buchweitz/Iyengar/Miller:2010} for more details.

\begin{ch}\label{dg_basics}
We let $(-)^\natural$ denote the functor that forgets the differential of a graded object. That is to say, if $M$ is a dg module over a dg algebra $A$, then $M^\natural$ is the underlying graded module over the graded algebra $A^\natural$. 

A dg $A$-module $M$ yields a pair of exact functors 
\[
-\lotimes_A M\colon \D(A)\to \D(\mathbb{Z})\quad \text{and}\quad \RHom_A(M,-)\colon \D(A)\to \D(\mathbb{Z})\,,
\]
and we write
\[
\Ext_A(M,-)\coloneqq \H(\RHom_A(M,-))\quad \text{and}\quad \Tor^A(M,-)\coloneqq \H(M\lotimes_A -)\,.
\]
As a special case, a map of dg algebras $\vp\colon A\to B$ yields an adjoint pair
\begin{equation}\label{e_tens_hom_ad}
-\lotimes_A B\colon \D(A)\to \D(B)\quad \text{and}\quad \vp_*\colon \D(B)\to \D(A)\,
\end{equation}
where the second functor is restriction of scalars along $\vp$. For $N$ in $\D(B)$, we typically write $N$ for $\vp_*(N)$. 
\end{ch}

\begin{ch}
\label{finiterep}
When $A$ is nonnegatively graded, there is a map of dg algebras $A\to~\H_0(A)$, and restriction of scalars makes each complex of $\H_0(A)$-module into a dg $A$-module. This will be used without further mention, especially applied to the map $E\to \widehat{R}$.

Assume that $A$ is a nonnegatively graded dg algebra and $\H_0(A)$ is a commutative noetherian ring. By \cite[Appendix~B.2]{Avramov/Iyengar/Nasseh/SatherWagstaff:2019}, for each $M$ in $\Df(A)$ there exists a semifree resolution $F\xra{\simeq} M$ over $A$ with 
\[
F^\natural\cong \bigoplus_{j=i}^\infty \shift^j (A^\natural)^{\beta_j}
\]
for some nonnegative integers $\beta_j $. Let $s=\sup\{j:\H_j(M)\neq 0\}$ and note that as $A$ is nonnegatively graded, the soft truncation 
\[
F'= \cdots \to 0 \to \coker \del^F_{s+1}\to F_{s-1}\to F_{s-2}\to \cdots
\]
is quasi-isomorphic to $M$ as a dg $A$-module. Moreover, $(F')^\natural$ is finitely generated as a graded $A^\natural$-module. 
 \end{ch}

\subsection{Cohomological support}\label{s_support}
In this subsection, we recall the definition of our main objects of study. The theory of support varieties over a local complete intersection ring was defined by Avramov in \cite{Avramov:1989}; see also \cite{Avramov/Buchweitz:2000b} for notable applications of these geometric objects. D.\@ Jorgensen extended this theory in \cite{Jorgensen:2002} to arbitrary local rings using intermediate hypersurface rings; see also \cite{Avramov/Iyengar:2018}. These theories of support are recovered by the cohomological supports in \cref{c_cohsupp}; see \cite[Section~5.2]{Pollitz:2021}.

\begin{ch}
\label{c_cohsupp}
By \cite[Section~2]{Avramov/Buchweitz:2000a}, $\S$ can be identified with a graded $k$-subalgebra of $\Ext_E(k,k)$. Hence for any dg $E$-module $N$, through the composition  pairing 
\[
 \Ext_E(k,N)\otimes_k\Ext_E(k,k)\to  \Ext_E(k,N)\quad \text{given by}\quad \alpha\otimes \beta\mapsto \alpha\beta\,,
\]
$ \Ext_E(k,N) $  is a graded $\S$-module. 
Moreover, for any  $M$ in $\Df(R)$ the graded $k$-space $\Ext_E(k,\widehat{M})$ is a finitely generated graded $\S$-module; see \cite[Proposition~3.2.5]{Pollitz:2019}. 
\end{ch}

\begin{ch}
The \textbf{cohomological support variety} of an object $M$ in $\Df(R)$ is
\[
\V_R(M)\coloneqq \supp_\S \Ext_E(k,\widehat{M})\subseteq \spec \S\,.
\]
\end{ch}

\begin{ch}\label{c_support_symmetry}
According to \cite[Theorem~6.1.6]{Pollitz:2021}, for any $M$ in $\Df(R)$ we have
\[
\V_R(M)= \supp_\S \Ext_E(k,\widehat{M}) = \supp_\S \Ext_E(\widehat{M},k)\,,
\]
this time using the natural left action of $\Ext_E(k,k)$ on $\Ext_E(\widehat{M},k)$ by Yoneda composition. In other words, we may compute cohomological support varieties using the Ext-modules with $\widehat{M}$ in either argument.

It follows from Nakayama's lemma that $\Ext_E(\widehat{M},k) = 0$ only when $M\simeq 0$. This in turn implies that 
\[
\V_R(M)=\varnothing\text{ if and only if } M\simeq 0\,.
\]
\end{ch}

\begin{rem}
The fact that $\V_R(M)$ is independent of the choice of Cohen presentation, and of the minimal generators for the defining ideal of $\widehat{R}$ in such a presentation, is dealt with by \cite[Theorem~6.1.2]{Pollitz:2021}. 
\end{rem}

\begin{ex}
\label{ex_characterization}
By \cite[Theorem~3.3.2]{Pollitz:2019},  $R$ is complete intersection if and only if $\V_R(R)=\{\bm{0}\}$, the origin. In fact, in  \cite[Theorem~6.1.6]{Pollitz:2021}, this equivalence was strengthened to say that $R$ is  complete intersection if and only if $\V_R(M)=\{\bm{0}\}$ for some finitely generated $R$-module $M$. 
\end{ex}

\begin{ex}
One always has $\V_{R}(k)=\spec \S$.
Indeed, $\Ext_E(k,k)$ is a finite rank free graded $\S$-module; cf.\@  \cite[Remark~3.2.6]{Pollitz:2021}.
\end{ex}

\begin{ex}
\label{ex_koszul}
Given a ring of central cohomology operators, one has control on the support of the \emph{Koszul objects} introduced in \cite{Avramov/Iyengar:2007} (see also \cite{Benson/Iyengar/Krause:2008}). In the setting  of the present article,  $\S$ is a ring of central cohomology operators on $\D(E)$ but not $\D(R)$, unless $R$ is complete intersection. Regardless, \cite[Theorem~3.3.4]{Pollitz:2019} established that for any nonzero $\zeta \in \S^d$, with $d$ any nonnegative even integer, there exists an object $L_\zeta$ in $\Df(R)$ which plays the role of a Koszul object in the sense that
\[
\V_R(L_\zeta)=\cV(\zeta)\,. 
\]
That is, each conical hypersurface in $\spec \S$ is realizable as the cohomological support of an object in $\Df(R)$. The objects $L_\zeta$ are defined as 
\[
L_\zeta \coloneqq\mathrm{cone}(k\xra{\tilde{\zeta}}\shift^dk)
\]
in $\Df(R)$, where $\tilde{\zeta}$ is a lift of $\zeta$ along the map
\[
\Ext_R(k,k)\cong \Ext_{\widehat{R}}(k,k)\to \Ext_E(k,k)\,;
\]
the fact this map is surjective is \cite[Theorem~2.3.2]{Pollitz:2019}.
Consequently, if $d>0$, then 
\[  \H_i(L_\zeta)=\begin{cases}
k & i=1,d\\
0 & \text{otherwise}\,. 
\end{cases}
\]
In \cref{golod}, we will see that these hypersurfaces are the only proper subsets of $\spec \S$ that are \emph{always} realizable regardless of $R$; see also \cref{ex_bej}. The complexes $L_\zeta$ are analogous to the  Carlson modules defined over group algebras; cf.~\cite[Section~1.10]{Benson:2017}.
\end{ex}

\begin{ex}
\label{ex_codepth}
Recall $R$ has an {\bf embedded deformation} if there exists a local ring $S$ and an element $f$ in the square of the maximal ideal of $S$ such that $R \cong S/(f)$.
By \cite[Theorem~6.3.5]{Pollitz:2021}, whenever $\codepth(R) \coloneqq \depth (Q) - \depth(R) \les 3$ we have
\[
\V_R(R)=\begin{cases}
\{\bm{0}\} & \text{if } R \text{ is complete intersection}\\
\spec \S & \text{if $R$ does not admit an embedded deformation}\\
\cV(\zeta) & \text{for some }\zeta\in \S^2, \text{otherwise}.
\end{cases}
\]
So when the codepth of $R$ is small, $\V_R(R)$ is always a linear space. In \cite[Example~5.4]{Briggs/Grifo/Pollitz:2022} an example was provided of a non-linear variety. Namely, 
\[
\V_R(R) = \cV(\chi_1\chi_5) = \cV(\chi_1)\cup \cV(\chi_5) \subseteq \spec k[\chi_1,\ldots,\chi_5]
\]
where $R=k[\![x,y,z,w]\!]/(x^2,xy,yz,zw,w^2)$. 
\end{ex}

\begin{ex}
\label{ex_bej}
One can always construct a finitely generated $R$-module $M$ whose support is contained in a hyperplane. To do this, fix a minimal generator $f$ of $I$ of minimal $\m$-adic order, and find a complete intersection ideal $J \supseteq I$ such that $f$ is also a minimal generator of $J$. By \cite[Lemma 4.2]{Briggs/Grifo/Pollitz:2022}, such a $J$ always exists, and by \cite[Lemma 3.6]{Briggs/Grifo/Pollitz:2022}, $\V_R(Q/J)$ is a linear space contained in the hyperplane determined by $f$.
\end{ex}

\subsection{Thick subcategories and levels}\label{s_level}
Here we recall a notion from \cite{Avramov/Buchweitz/Iyengar/Miller:2010} that  counts the number of mapping cones needed to build one object from another in a  triangulated category; see also  \cite[Section~2]{Bondal/VanDenBergh:2003} or \cite[Section~3]{Rouquier:2008}. 

\begin{ch}
\label{c_level}
Let $\mathsf{T}$ be a triangulated category with suspension functor $\shift,$ and fix an object $N$ in $\T$. We write $\thick_\T^0 N=\{0\}$ and $\thick^1_\T N$ for the smallest full additive subcategory of $\T$ containing $\shift^iN$ for all $i$ and closed under direct summands. Then inductively, $\thick_\T^n N$ is the smallest full subcategory of $\T$ closed under direct summands that contains all objects $M$ for which there is an exact triangle 
\[
M'\to M\to M''\to 
\]
with $M'$ in $\thick_\T^1N$ and $M''$ in $\thick_\T^{n-1}N.$ 
We also set 
\[
\thick_\T N\coloneqq \bigcup_{n=0}^\infty \thick_\T^n N\,,
\]
and note that $\thick_\T N$ is exactly the thick closure of $N$ in $\T$; that is, $\thick_\T N$ is the smallest \emph{triangulated} subcategory of $\T$ that is closed under direct summands. For $M$ in $\T$, we define $\level_\T^N M$ to be the smallest integer $n$ such that $M$ belongs to $\thick_\T^nN$; if no such integer exists then $\level_\T^N M=\infty.$
\end{ch}

\begin{notation}
Let $A$ be a dg algebra. 
 When calculating level in $\D(A)$ we will write $\level_A$ instead of $\level_{\D(A)}$.
\end{notation}

\begin{ex}
Recall an $R$-complex $M$ is \textbf{perfect} if it is in $\thick_R R$. That is, $M$ is quasi-isomorphic to a finite free $R$-complex; the latter is a bounded complex of finite rank free $R$-modules. 
\end{ex}

\begin{ex}
In the notation of \cref{ex_koszul}, for each $\zeta\in \S^d$ with $d>0$, satisfies $\level_{R}^k(L_\zeta)=2.$ Indeed, by definition $\level_{R}^k(L_\zeta)\leqslant 2$. Since $\V_R(N)=\spec \S$ for all nonzero $N$ in $\thick_R^1 k$, the desired equality follow the already noted strict containment  $
\V_R(L_\zeta)\subsetneq \spec \S;$ cf.\@ \cref{ex_koszul}. 
\end{ex}

\begin{ch}\label{rem_thick_support}
   Let $R$ be a local ring and let $M,N$ be in $\Df(R)$. By \cite[3.3.2]{Pollitz:2019} if $M$ is in $\thick_R N$ then $\V_R(M)\subseteq \V_R(N)$. It follows that the cohomological support varieties may be thought of as invariants of thick subcategories of $\Df(R)$; see \cref{cor_thick_intersection} for an application of this idea.
\end{ch}


\begin{ch}
\label{c_levelsgodown}
Given an exact functor of triangulated categories $\t\colon \T\to \mathsf{S}$, it follows directly from the definition of level that 
\[
\level_{\mathsf{S}}^{\t N}{\t M}\les \level_\T^NM
\]
for any pair of objects $M,N$ in $\T$. If $\mathsf{t}$ admits a left inverse then equality holds. 

In particular, for any map $\vp\colon A\to B$ of dg algebras, and for any $M,M'$ in $\D(A)$ and $N,N'$ in $\D(B)$, the exact functors in \cref{dg_basics}\cref{e_tens_hom_ad} satisfy inequalities
\[
\level_B^{M\lotimes_A B} (M'\lotimes_A B)\les \level_A^MM'\quad \text{and}\quad \level_A^NN'\les \level_B^NN'\,.
\]
\end{ch}

\subsection{Semifree dg algebras and the LS category}
\label{subsec_dg}
 We recall some necessary facts on semifree dg algebras in the section. Much of the background needed is contained in \cite{Avramov:2010,Briggs:2018}. 

\begin{ch}
A dg  algebra over a commutative ring $T$ is called \textbf{semifree} if it is concentrated in nonnegative homological degrees and, as a graded $T$-algebra, it is isomorphic to the free strictly graded commutative algebra $T[X]$ on a graded set of variables $X = X_1\cup X_2 \cup \cdots$ (the elements of $X_n$ having degree $n$). Explicitly, $T[X]$ is the tensor product of the exterior algebra over $T$ on $X_{\text{odd}}$ with the symmetric algebra over $T$ on $X_{\text{even}}$. 
\end{ch}

\begin{ex}\label{eq_koszul_semi_free}
The Koszul complex $E$ on $\f$ over $Q$ is semifree with variables $X_1=~\{e_1,\ldots,e_n\}$ and $X_i=\varnothing$ for $i\geqslant 2$.
\end{ex}

\begin{ch}
\label{c_semifree} 
A \textbf{semifree resolution} of $\widehat{R}$ over $Q$ is a quasi-isomorphism $Q[X]\to \widehat{R}$ of dg $Q$-algebras where $Q[X]$ is semifree over $Q$. 
Such resolutions may be constructed inductively using the method of Tate, successively adjoining variables to kill homology classes (but using polynomial variables instead of divided power variables); see \cite{Tate:1957} or \cite{Avramov:2010} for details. By \cite[Lemma 7.2.2]{Avramov:2010}, when this resolution is constructed by adjoining the minimal possible number of variables at each stage, the obtained dg-algebra $Q[X]$ has a {\bf decomposable} differential, meaning that $\partial(x)\in (\m, X^2)$ for any $x\in X$. Such a resolution is called a {\bf minimal model} for $R$; by \cite[Proposition 7.2.4]{Avramov:2010} minimal models are unique up to isomorphism.
\end{ch}

\begin{ch}
\label{c_deviations} 
Since $Q[X]$ is well-defined up to isomorphism, the numbers $\varepsilon_i(R) = |X_{i-1}|$ for $i\geqslant 2$ are well-defined invariants of $R$, known as the {\bf deviations} of $R$; by convention $\varepsilon_1(R)= \dim(Q)$ is the embedding dimension of $R$. The Koszul complex $E$ is exactly the first step in the construction of the minimal model of $R$, and therefore $\varepsilon_2(R)=n$, the minimal number of generators of $I$.
\end{ch}

\begin{ch}
\label{c_LScategory}
Lusternik-Schnirelmann category is a numerical invariant having its origins in topology that extends Loewy length to the realm of dg algebras.

Let $A=Q[X]$ be a minimal semifree dg algebra as above, with maximal ideal $\m_A=(\m,X)$. By definition, $\cat(A)$ is the smallest integer $m$ such that there is a dg algebra $B$ and a diagram
    \begin{equation}
        \label{eq_LS_diagram}
    \begin{tikzcd}
    A \ar[r,"\iota"] & B \ar[d,"\simeq","\tau"'] \ar[r,"\rho"]  & A \\
    & A/\m_A^{m+1}
    \end{tikzcd}
    \end{equation}
    such that $\rho\iota={\rm id}_A$ and $\tau\iota=\pi$, 
    the natural projection from $A$ to $A/\m_A^{m+1}$.  If there is no such integer, then $\cat(A)=\infty$. This definition first appears in \cite{Felix/Halperin:1980}, where it is given as a characterization of the rational LS category of a simply connected space $Y$, taking $A$ to be the minimal Sullivan model associated to $Y$.

   We remark that (\ref{eq_LS_diagram}) implies that the inclusion $\m_A^{m+1} \hookrightarrow \m_A$ is a nullhomotopic chain map; we will use this particular consequence of the definition to prove \cref{th_radical_support}.
\end{ch}

\begin{ch}
\label{c_LScat_functors}
Adopting the notation from \cref{c_LScategory}, restriction along the maps in 
(\ref{eq_LS_diagram}) yields exact functors of triangulated categories 
    \begin{equation}
        \label{embedding_lemma}
    \begin{tikzcd}[column sep = 20mm]
    \D(A) \ar[r,shift left = 1mm, "F= (\tau_*)^{-1}\rho_*"] & \ar[l,shift left = 1mm, "\pi_*"] \D(A/\m_A^{\cat(A)+1}) \quad \text{ with } \quad \pi_*F={\rm id}\,.
    \end{tikzcd}
    \end{equation}
This observation will be used in the proof of \cref{l_levelvscat}.
\end{ch}

\begin{ch}\label{c_fiveauthorbound}
We associate an important numerical invariant to any local ring $R$  as in \cite{Felix/Halperin/Jacobsson/Lofwall/Thomas:1988} by passing to a semifree resolution and using the Lusternik-Schnirelmann category for dg algebras.

To be precise, let $Q[X]$ be a minimal model for $R$ as in \ref{c_semifree}, and set \[A\coloneqq k[X]=k\otimes_Q Q[X]\,.\] Then the \textbf{LS category} of $R$ is by definition the integer
\[
\cat(R) \coloneqq \cat(A)\,.
\]
A priori the LS category may be infinite, but by \cite[Lemma 2.3]{Bogvad/Halperin/1986} we have
\begin{equation}
\label{codepth_inequality}
\cat(R) \leqslant \codepth(R)\,,
\end{equation}
and in particular $\cat(R)$ is always finite. For artinian rings LS category is, by \cite[Proposition 2.4]{Bogvad/Halperin/1986}, also bounded by the Loewy length of $R$:
\[
\cat(R) \leqslant \ell\ell(R)\,.
\]
This inequality is a special case of the \emph{mapping theorem}; see  \cite[Theorem I]{Felix/Halperin:1980} and \cite[Theorem 20]{Briggs:2018}. Moreover, the inequality in \cref{codepth_inequality} is strict when $R$ is not complete intersection. 
\end{ch}

\subsection{The homotopy Lie algebra}\label{s_HLA_background}
To each local ring $R$, one can associate a graded Lie algebra $\pi^*(R)$ over $k$, known as the \textbf{homotopy Lie algebra of $R$}, and its radical subspace
\[
\rad(\pi^*(R)) \subseteq \pi^*(R)\,.
\]
\begin{ch}
Fix a minimal model $Q[X]$ for $R$ as defined in \cref{c_semifree}.  Following \cite{Avramov/Halperin:1987} we set
\[
\pi^i(R)\coloneqq (kX_{i-1})^\vee \quad \text{for}\quad i\geqslant 2\,,
\]
where $(-)^\vee$ denotes $k$-space duality. Often we will think of elements of $\pi^i(R)$ as functionals on $kX$, taking the value zero on $X_j$ when $j\neq i-1$.  We note that $\rank_k\pi^i(R)=\varepsilon_i(R)$ by definition (see \cref{c_deviations}), and therefore one can think of  $\pi^*(R)$ as an algebraic object that enriches the deviations of $R$.

A Lie bracket can be defined on this graded vector space using the fact that the differential of $Q[X]$ is decomposable using the following recipe. Given $\alpha \in \pi^i(R)$ and $\beta\in \pi^j(R)$, we need to define $[\alpha,\beta]\in \pi^{i+j}(R)=( X_{i+j-1})^\vee$. To do that, let $x\in X_{i+j-1}$ and suppose that \[\partial(x)=\sum_{pq} a_{pq} y_pz_q \mod (\m + (X)^3)\,,\] with $a_{pq}\in Q$ and $y_p,z_q\in X$. The bracket $[\alpha,\beta]$ is then determined by
\[
[\alpha,\beta]( x) \coloneqq \sum_{pq} \overline{a}_{pq}\big( (-1)^{i+1+ij} \alpha( y_p) \beta( z_q) +(-1)^{j} \beta(y_p) \alpha( z_q)\big)
\]
with $\overline{a}_{pq}$ denoting the class of ${a}_{pq}$ in $k$; cf.~\cite[page 175]{Avramov/Halperin:1987}. This makes $\pi^{\ges 2}(R)$ into a graded Lie algebra over $k$. 

We have only defined the homotopy Lie algebra in degrees $2$ and above. The full homotopy Lie algebra, including $\pi^1(R)$, can be defined along similar lines, cf.~\cite{Avramov:2010} or \cite{Briggs:2018}.

When $k$ has characteristic $2$ a graded Lie algebra should possess a squaring operation on odd degree elements, in addition to its bracket. Moreover in characteristic $2$ or $3$ care has to be taken to state the Jacobi identity in full. These technical points do not affect the constructions of this paper, and so we refer to \cite[Remark~10.1.2]{Avramov:2010} for details.
\end{ch}

\begin{ch}
\label{c_identification}
The homotopy Lie algebra in degree two admits a simple description. Keeping in mind that the Koszul complex $E=Q[X_1]$, see  \cref{eq_koszul_semi_free}, is  the first step in the construction of $Q[X]$, we a have canonical isomorphism
\[
    \pi^2(R) = (kX_1)^\vee = \mathrm{span}_k\{e_1^\vee,\ldots,e_n^\vee\} \xrightarrow{\ \cong\ } (I/\m I)^\vee\,,\ \quad e_i^\vee\mapsto f_i^\vee\,.
\]
\end{ch}

\begin{ch}
For any element $\alpha \in \pi^i(R)$ we consider the adjoint action
\[
\ad(\alpha)\coloneqq[\alpha, -] \colon \pi^*(R) \to \pi^{*+i}(R)\,.
\]
We say that $\alpha$ is {\bf central} if $\ad(\alpha)=0$, and {\bf radical} if $\ad(\alpha)^p=0$ for some $p$. The radical of $\pi^*(R)$ is the set of all radical elements
\[
\rad^*(\pi^*(R)) \coloneqq \{ a\in \pi^*(R) \mid \ad(\alpha)^p=0 \text{ for some }p\}\,.
\]
It follows from the results of  \cite{Felix/Halperin/Jacobsson/Lofwall/Thomas:1988} that $\rad^*(\pi^*(R))$ is a graded Lie ideal in $\pi^*(R)$; in fact it is the maximal solvable ideal. We will be especially interested in radical elements living in degree $2$, which we will denote by $\rad^2(\pi^*(R))$. 
\end{ch}

\section{Lower bounds on cohomological support}\label{s_bounds}

 The goal here is to show there is a uniform lower bound on $\dim \V_R(M)$, for all nonzero $M$ in $\Df(R)$, based on  invariants of $R$; see \cref{t_bound}.

First we use a well-traversed bridge from homological algebra over $R$ to homological algebra over an exterior algebra, exploited in \cite{Avramov/Buchweitz/Iyengar/Miller:2010,Avramov/Iyengar:2010,Iyengar/Pollitz/Sanders:2022,Liu/Pollitz:2021} to name a few. This is a dg version of the celebrated BGG correspondence~\cite{Bernstein/Gelfand/Gelfand:1978}; see \cite[Section~7]{Avramov/Buchweitz/Iyengar/Miller:2010}. 
\begin{construction}
\label{t functor} 
We use the notation from \ref{notation_s_support}, and in particular $E$ is the Koszul complex on $\f$ over $Q$. We also write $\Lambda\coloneqq E\otimes_Q k$, the  exterior algebra over $k$ on $n$ variables of homological degree $1$. Consider the functor
\[
\t\colon\D(E)\longrightarrow \D(\Lambda)
\]
given by  $\t\coloneqq -\lotimes_Q k$. That is, for an object $N$ in $\D(E)$ we have
\[
\t (N)\simeq F\otimes_Q k\,,
\]
where $F\xra{\simeq} N$ is a quasi-isomorphism of dg $E$-modules, with $F^\natural$ free as a $Q$-module.  
Note $\t$ restricts to a functor $\Df(E)\to \Df(\Lambda)$, also denoted $\t$, since each object of $\Df(E)$ is perfect over $Q$.

\label{c_bgg}
The second functor of interest is the equivalence from \cite{Avramov/Buchweitz/Iyengar/Miller:2010}:
\[
\h\colon \Df(\Lambda)\xra{\ \equiv\ } \Df(\S)
\]
given by $\h=\RHom_\Lambda(k,-)$. In $\D(\S)$ there is an isomorphism $
\h(k)\simeq \S$, 
and so, using \cref{c_levelsgodown}, for any dg $\Lambda$-module $N$, the following equality is satisfied
\[
\level_\Lambda^k (N)=\level_\S^\S(\h N)\,.
\]
\end{construction}

\begin{lemma}
\label{tk_trivial}
In the notation from \cref{t functor}, $\mathsf{t}(k)\simeq \bigwedge \shift k^{\varepsilon_1(R)}$ in $\D(\Lambda)$.
\end{lemma}

\begin{proof}
Since $Q$ is regular, there is a quasi-isomorphism of dg $Q$-algebras $K^Q\xra{\simeq}k$. The dg algebra $K^Q$ obtains a dg $E$-algebra structure by choosing a lift of the canonical map $E\to k$ 
along this quasi-isomorphism. Since $\f\subseteq \m^2$, any such lift $E\to K^Q$ will factor in degree $1$ as
\[
E_1\to \m K_1^Q\to K_1^Q\,;
\]
   see, for example, \cite[Section~2.2]{Pollitz:2019}. By the definition of $\mathsf{t}$, the action of $\Lambda_1 = E_1\otimes_Qk$ on $\mathsf{t}(k)\simeq  k\otimes_QK^Q$  factors through the zero map \[k\otimes_Q \m K_1^Q\to k\otimes_QK_1^Q\,,\] and therefore $\mathsf{t}(k)$ is isomorphic to a direct sum of shifted copies of $k$ in $\D(\Lambda)$.
\end{proof}

\begin{proposition}
\label{p_koszulbggsupport}
For each $M$ in $\Df(R)$, the following equalities are satisfied:
\[
\V_R(K^M)=\V_R(M)=\supp_\S \H(\h\t M)\,.
\]
\end{proposition}
\begin{proof}
First note 
\[
\Ext_E(k,\widehat{K^M})\cong \Ext_E(k,\widehat{M})\otimes_k \bigwedge \shift k^{\varepsilon_1(R)}
\]
as graded $\S$-modules. Indeed, $K^M$ is an iterated mapping cone on a minimal generating set for $\m$ starting from $M$, and so the isomorphism holds by induction; see also \cite[Lemma~3.2.4]{Pollitz:2019}. From this isomorphism, the first equality holds. The second equality is from \cite{Iyengar/Pollitz/Sanders:2022}.
\end{proof}

\begin{chunk}
\label{c_closedsubsets}
Fix  a closed subset $\mathcal{U}$ of  $\spec \S$. The dimension of $\mathcal{U}$ is  
\[
\dim \mathcal{U}\coloneqq\dim \S/\mathcal{I}
\]
where $\mathcal{U}=\cV(\mathcal{I})$ for a homogeneous ideal $\mathcal{I}$ of $\S$, and the right-hand side is the Krull dimension of $\S/\mathcal{I}.$  The codimension of $\mathcal{U}$ is 
\[
\codim\mathcal{U}\coloneqq\dim \S-\dim \mathcal{U}\,,
\]
and the height of $\mathcal{I}$ is 
\[
\height \mathcal{I}\coloneqq\inf\{n \ges 0: \p_0\subseteq\ldots \subseteq\p_n\text{ in }\spec \S\text{ with }\p_n\supseteq \mathcal{I}\}\,.
\]
Since $\S$ is Cohen-Macaulay, by \cite[Corollary~2.1.4]{Bruns/Herzog:1998} we have
\[
\codim\mathcal{U}=\height \mathcal{I}\,.
\]
\end{chunk}

\begin{chunk}
\label{c_newintersection}
For any dg $\S$-module $X$, the \emph{dg new intersection theorem} states: 
\[
\level_\S^\S X\ges \height(\ann_\S \H(X))+1\,.
\]
This is stated and proved in \cite[Theorem~5.1]{Avramov/Buchweitz/Iyengar/Miller:2010} when is $\S$ is a graded algebra over a field,  as it is in this article. However, we record that in light of the work of Andr\'{e}~\cite{Andre:2018}, see also \cite{Bhatt:2018}, the assumption that this graded ring contain a field is no longer necessary. 
\end{chunk}

\begin{lemma}
\label{l_levelvscat}
Let $(R,\m,k)$ be a local ring and let $K^R$ be the Koszul complex on a minimal generating set for $\m$. For  $N$  in $\Df(K^R)$, we have
\[
\level_{K^R}^k N\les \cat(R)+1\,.
\]
\end{lemma}

\begin{proof}
First, note that the level and LS-category in the statement are unchanged when we pass to the completion, so we can reduce to the case when $R$ is complete. Fix a minimal Cohen presentation $Q\to R$ and a minimal model $Q[X]\xra{\simeq} R.$ The quasi-isomorphisms of dg algebras 
\[
K^R \xleftarrow{\simeq} K^Q\otimes_Q Q[X]\xra{\simeq} A\coloneqq k\otimes_Q Q[X]
\]
define an equivalence of categories $\D(K^R)\equiv \D(A)$; under this equivalence $k$ is sent to $k$ as the quasi-isomorphisms above respect the augmentation maps to $k$.  Hence, by \cref{c_levelsgodown} it suffices to show  $\level_A^k M\les \cat(R)+1$ for each $M$ in $\Df(A)$.

Set $C=A/(X)^{\cat(R)+1}$ and consider the functor
\[
\D(A)\to \D(C)\,
\]
described in \Cref{c_LScat_functors}. By \cref{c_LScat_functors}\cref{embedding_lemma}, this functor admits a left inverse, and so we have further reduced the problem to showing that for each $M$ in $\Df(C)$ there is an inequality
\[
\level^k_C M\les \cat(R)+1.
\]

To this end, as $M$ is in $\Df(C)$ we can assume the underlying module $M^\natural$ obtained from $M$ by forgetting the differential is finitely generated over the graded algebra $C^\natural$; see \cref{finiterep}. Consider the sequence of dg $C$-modules
\[
0=(X)^{\cat(R)+1}M\subseteq \ldots \subseteq (X)^2M\subseteq (X)M\subseteq M\, .
\]
The fact that these are dg $C$-submodules of $M$ uses that $Q[X]$ was chosen as a minimal model, and so $C$ satisfies $\del(X)\subseteq (X)^2$. 
Since $M^\natural$ is finitely generated over $C^\natural$, each quotient
\[(X)^iM/(X)^{i+1}M\]
is a finite dimensional graded $k$-space and consequentially the desired inequality follows.
\end{proof}

We now arrive at the main result of the section. 
The lower bound in it takes inspiration from  \cite[Theorem~7]{Avramov/Buchweitz/Iyengar/Miller:2010}; see also \cite[Theorem~9.1]{Avramov/Iyengar:2010}.

\begin{theorem}
\label{t_bound}
Let $R$ be a local ring with residue field $k$. If $M$ is a nonzero object in $\Df(R)$, then the following inequalities are satisfied:
\[
\cat(R)+1 \ges \level_{K^R}^k (K^M)\ges \varepsilon_2(R)-\dim \V_R(M)+1\,.
\]
\end{theorem}
\begin{proof}
The first inequality is \cref{l_levelvscat}. For the second inequality, set $N=K^M$, and consider the inequalities \begin{align*}
\level^k_{K^R}(N) & \ges \level^k_R(N) & \textrm{by \cref{c_levelsgodown}, restricting along } R\to K^R\\
& \ges \level^{\t k}_\Lambda (\t N) & \textrm{by applying } \t  \textrm{ and \cref{c_levelsgodown}}\\
&= \level^{k}_\Lambda (\t N) & \textrm{using  \cref{tk_trivial}} \\
&= \level^{\S}_\S (\h\t N) & \textrm{using \cref{c_bgg}}\\
& \ges \text{height}( \ann_\S \H(\h\t N))+1 & \textrm{by \cref{c_newintersection}}.
\end{align*}

Also, the following equalities are satisfied
\begin{align*}
\height( \ann_\S \H(\h\t N)) &=\codim \cV(\ann_\S \H(\h\t N)) & \textrm{by \cref{c_closedsubsets}}\\
&=\dim \S-\dim \V_R(M) & \\
&=\varepsilon_2(R)-\dim \V_R(M) & \textrm{by \cref{c_deviations}}
\end{align*}
where the second equality comes from \cref{c_hsupport} and \cref{p_koszulbggsupport} since $\H(\h\t N)$ is a finitely generated $\S$-module. 

Combining the two strings of inequalities completes the proof.
\end{proof}

The {\bf complete intersection defect} of $R$, introduced in \cite{Avramov:1977} and implicitly appearing in \cite{KiehlKunz}, is defined as
\[
\cid(R)\coloneqq\varepsilon_2(R)-\varepsilon_1(R)+\dim R\,,
\]
or alternatively, $\cid(R)=\rank_k (I/\m I) - \height(I)$ \cite[Lemma 7.4.1]{Avramov:2010}. This is a nonnegative integer that measures how close $R$ is to being complete intersection, in particular, $R$ is complete intersection if and only if $\cid(R)=0;$ see \cite[Section 1]{KiehlKunz} and \cite[(3.2.3)]{Avramov:1977}.
Showing this invariant is additive along the base, source and fiber of a flat map was fundamental in Avramov's solution to the localization problem for complete intersection rings \cite{Avramov:1975,Avramov:1977}.

\begin{corollary}
\label{cor_bound}
Every nonzero object $M$ in $\Df(R)$ satisfies
\[
\dim \V_R(M) \ges \varepsilon_2(R)-\cat(R)\,.
\]
In particular, if $R$ is Cohen-Macaulay but not complete intersection, then 
\[\dim \V_R(M)> \cid(R)>0\]
for each nonzero object $M$  in $\Df(R).$
\end{corollary}
\begin{proof}
The first inequality is immediate from \cref{t_bound}. The second inequality now follows from the first and \cref{codepth_inequality} in \cref{c_fiveauthorbound}.
\end{proof}

There is no uniform bound on the difference between $\dim \V_R(R)$ and $\cid(R)$ when $R$ is not complete intersection. 
\begin{example}
\label{example dim vs cid}
Let $R=k[x_1,\ldots,x_i]/(x_1,\ldots,x_i)^2$ for $i\ges 2$. From the definition of complete intersection defect, 
\[
\cid(R) = \binom{i+1}{2} - i\,.
\]
On the other hand, $R$ is Golod, and therefore by \cref{golod} it follows that for any nonzero $M$ in $\Df(R)$:
\[
\dim \V_R(R) =\binom{i+1}{2} \quad \text{and} \quad \dim \V_R(M)\ges \binom{i+1}{2}-1\,,
\]
and in particular $
\dim \V_R(R) - \cid(R) = i$.

\end{example}

\begin{chunk}
The \textbf{realizability question} asks: \emph{Can every subset of $\spec \S$ be realized as $\V_R(M)$ for some $M$ in $\Df(R)$?} 
This has been studied in several contexts, see for example \cite{Avramov:1989,Bergh,Burke/Walker:2015,Carlson:1984,Erdmann:2004,Suslin/Friedlander/Bendel:1997}.
It was shown in \cite{Bergh} that the realizability problem has a positive answer when $R$ is complete intersection; see also \cite{Avramov/Iyengar:2007}. In \cite{Pollitz:2021}, it was shown that if $R$ is not complete intersection, then $\{\bm{0}\}$ is never realizable as $\V_R(M)$ when $M$ is a finitely generated $R$-module. However, it was was asked in \cite{Pollitz:2021} whether the solution to the realizability problem is positive regardless of whether $R$ is complete intersection -- in particular, whether $\{\bm{0}\}$ can be realized by some object in $\Df(R)$.
\end{chunk}

\cref{example dim vs cid} provides a class of rings for which the realizability problem has a negative answer in a rather drastic fashion. We also have the following for Cohen-Macaulay rings; it is an immediate consequence of \cref{cor_bound}.

\begin{corollary}
\label{cor_non_realizability}
If $R$ is a Cohen-Macaulay ring, there exists $M$ in $\Df(R)$ with $\V_R(M) = \{\bm{0}\}$ if and only if $R$ is complete intersection. 
\end{corollary}

\begin{remark}
\label{refined realizability}
In light of \cref{cor_bound}, a more sensible version of the realizability question is the following problem: \emph{What subsets of $\spec \S$ of codimension at most $\cat(R)$ are realizable as $\V_R(M)$ for some $M$ in $\Df(R)$?}
\end{remark}

\begin{remark}\label{rem_ci_thick_subcats}
   When $R$ is complete intersection the Koszul complex $K^R$ belongs to every thick subcategory of $\Df(R)$ by the work of Dwyer, Greenlees, and Iyengar \cite{Dwyer/Greenlees/Iyengar:2006b}, and so there is a unique minimal thick subcategory (see also~\cite[Theorem~5.2]{Pollitz:2019} for the converse). In contrast, when $\{\bm{0}\}$ is not realizable (as in \cref{cor_non_realizability}), the cohomological support varieties take values in $\Spec^* \S \smallsetminus \{\bm{0}\} = \mathrm{Proj} \S$, and we note next that this implies that the lattice of thick subcategories of $\Df(R)$ exhibits ``non-local'' behaviour.
\end{remark}

\begin{proposition}
\label{cor_thick_intersection}
If $R$ is a local ring such that $\{\bm{0}\}$  is not the cohomological support of any object in $\Df(R)$ then there are infinitely many nonzero thick subcategories of $\Df(R)$ having pairwise zero intersection. 
\end{proposition}

\begin{proof}
Given $X\subseteq \Spec^*\S$   we write $\mathsf{T}_X$ for the full subcategory of $\Df(R)$ containing those $M$ with $\V_R(M)\subseteq X$; by \cref{rem_thick_support} this is a thick subcategory of $\Df(R)$.

Pick any nonconstant homogeneous element $\zeta_1 \in \S$ and let $X_1$ be a minimal closed subset of $\mathcal{V}(\zeta_1)$ such that $\mathsf{T}_{X_1}\neq (0)$. Such an $X_1$ exists because $L_{\zeta_1}\in \mathsf{T}_{\mathcal{V}(\zeta_1)}$ by \cref{ex_koszul}, and $X_1\neq \{\bm{0}\}$ by hypothesis.

We continue to construct $X_1,X_2,\ldots$ inductively. Assume that $X_1,\ldots,X_{n}$ are all minimal with the property that $\mathsf{T}_{X_i}$ are nonzero, and that $\mathsf{T}_{X_i}\cap \mathsf{T}_{X_j}=(0)$ when $i\neq j$. It follows again that each $X_i\neq \{\bm{0}\}$, and therefore we may choose a nonconstant homogeneous $\zeta_{n+1}\in \S$ such that $X_i\not\subseteq\mathcal{V}(\zeta_{n+1})$ for all $i\leqslant n$, and we may choose a minimal $X_{n+1}\subseteq \mathcal{V}(\zeta_{n+1})$ for which $\mathsf{T}_{X_{n+1}}\neq(0)$. To conclude, we note that $\mathsf{T}_{X_{n+1}}\cap \mathsf{T}_{X_{i}}\subseteq\mathsf{T}_{X_{i}\cap X_{n+1}} =(0)$ for all $i\leqslant n$, since $X_{i}\cap X_{n+1}$ is strictly contained in $X_i$ by construction.
\end{proof}

   We expect that the Cohen-Macaulay hypothesis in \cref{cor_non_realizability} is unnecessary, and therefore \cref{cor_thick_intersection} should apply to all rings that are not complete intersection.    In the case of Golod rings we can be more precise; see \cref{cor Lzeta intersection golod}.

\section{Upper bounds on cohomological support}\label{s_HLA}

 In this section we prove \cref{th_radical_support}, connecting the cohomological support varieties with the homotopy Lie algebra.  
 This result, paired with \cref{t_bound}, has several applications that we provide after the proof of \cref{th_radical_support}. 
 
 As in \cref{notation_s_support} and throughout, $R$ is a local ring with a Cohen presentation $\widehat{R}=Q/I$, where $I$ is minimally generated by $\f=f_1,\ldots,f_n$, with the corresponding Koszul complex $E$ and ring of cohomology operators $\S$.

 We also use background material from \cref{s_support,s_HLA_background}. In particular, in the notation of \ref{c_identification}, there is a canonical isomorphism  between the homotopy Lie algebra in degree $2$ and the space of cohomology operators:
 \[
    \pi^2(R) = \mathrm{span}_k\{e_1^\vee,\ldots,e_n^\vee\} \xrightarrow{\ \cong\ } \S^2 = \mathrm{span}_k\{\chi_1,\ldots,\chi_n\} \,,\ \quad e_i^\vee\mapsto \chi_i\,,
\]
where, recall, $(-)^\vee$ denotes $k$-space duality. 
In what follows, for an element \[\alpha= a_1 e_1^\vee + \cdots + a_n e_n^\vee\in \pi^2(R)\,,\] we will write $\chi_\alpha \coloneqq a_1 \chi_1+ \cdots + a_n \chi_n$ for the corresponding element of $\S^2$.
 
 In the next theorem, for a conical subset $\mathcal{U}$ of $\spec \S$, we write $\mathrm{span}_k \mathcal{U}$  for the smallest closed subset containing  $\mathcal{U}$ defined by linear forms in $\S$. When $k$ is algebraically closed, upon regarding $\mathcal{U}$ as a cone in the affine $n$-space $\mathbb{A}_k^n$, then $\mathrm{span}_k\mathcal{U}$ is the $k$-subspace spanned by $\mathcal{U}$.

\begin{theorem}\label{th_radical_support}
Under the isomorphism $\pi^2(R) \xra{\ \cong\ } \S^{2}$, radical elements of $\pi^2(R)$ are sent to linear polynomials vanishing on the subvariety $\V_R(R)\subseteq \spec \S$. In other words, the isomorphism  restricts to an embedding
\[
\rad^2(\pi^*(R)) \hookrightarrow \left\{ \chi\in\S^{2} \mid \V_R(R)\subseteq \mathcal{V}(\chi) \right\}\,.
\]
In particular
\[
\rank_k \rad^2(\pi^*(R)) \leqslant \varepsilon_2(R) - \rank_k (\mathrm{span}_k \V_R(R))\,.
\]
\end{theorem}

\begin{proof}
Consider a radical element $\alpha\in \pi^2(R)$. By definition, we have an inclusion $\V_R(R) \subseteq \mathcal{V}(\chi_\alpha)$ exactly when $\chi_\alpha$ acts nilpotently on  $\Ext_E(\widehat{R},k)$; cf.\@ also  \cref{c_support_symmetry}.

We use the theory of semifree dg algebra resolutions from \ref{c_semifree}, taking $Q[X]$ to be a minimal model for $\widehat{R}$, with  $E=Q[X_1]$. We recall that $\Ext_E(\widehat{R},k)\cong \Tor^E(\widehat{R},k)^\vee$, see \cite[4.3]{Pollitz:2021}, and in what follows we will describe the action of $\chi_\alpha$ in terms of this duality. As $Q[X]$ is semifree as a dg $E$-module we have
	\[
	\Ext_E(\widehat{R},k) \cong \H\!\big(\Hom_E(Q[X],k)\big) \cong \H\!\big(\Hom_k(Q[X]\otimes_Ek,k)\big) \cong (\H(k[X_{\geqslant 2}]))^\vee\,,
	\]
	where  $k[X_{\geqslant 2}]\coloneqq Q[X]\otimes_Ek$. It is therefore equivalent to show that $\chi_\alpha^q\H(k[X_{\geqslant 2}])=0$ for some $q$.
	
	At this point we need to recall how the cohomology operator $\chi_\alpha$ may be computed using divided powers; cf.\@  \cite[Section~6]{Avramov:2010} and \cite[Chapter~1]{Gulliksen/Levin:1969} for background on the latter. Starting from $k[X]=Q[X]\otimes_Qk$  construct the extension $k[X]\langle {y_1, \ldots, y_n} \rangle$, where the $y_i$ are free divided power variables with $\partial(y_i) = e_i$. Using \cite[Proposition 6.1.7]{Avramov:2010} factoring by the  $e_i$ and $y_i$ yields a quasi-isomorphism
	\[
	k[X]\langle {y_1, \ldots, y_n} \rangle \xrightarrow{\ \simeq\ } k[X_{\geqslant 2}]\,.
	\]
For each $i$ there is a unique derivation $\frac{d}{dy_i}$ on $k[X]\langle {y_1, \ldots, y_n} \rangle$, respecting the divided power structure, such that $\frac{d}{dy_i}(y_i)=1$, while $\frac{d}{dy_i}(y_j)=0$ for $i\neq j$ and $\frac{d}{dy_i}(X)=0$; see \cite[Section 6.2]{Avramov:2010}. Writing 
\[\frac{d}{dy_\alpha} = a_1 \frac{d}{dy_1}+\cdots +a_n \frac{d}{dy_n}\]
there is by \cite[Proposition~2.6]{Avramov/Buchweitz:2000a} a commutative diagram
	\[
\begin{tikzcd}
	    \H_*\!\big(k[X]\langle {y_1, \ldots, y_n} \rangle \big) \ar[d,"\H_*\left(\frac{d}{dy_\alpha}\right)"'] \ar[r,"{\cong}"] &  \H_*\!\big(k[X_{\geqslant 2}]\big) \ar[d,"{\chi_\alpha}"]\\
\H_{*-2}\!\big(k[X]\langle {y_1, \ldots, y_n} \rangle \big) \ar[r,"{\cong}"] & \H_{*-2}\!\big(k[X_{\geqslant 2}]\big)\,.
	\end{tikzcd}
	\]
In other words, $\frac{d}{dy_\alpha}$ induces the cohomology operator $\chi_\alpha$ on $\Tor^E(\widehat{R},k)$.	

We can go further and identify $\chi_\alpha$ as a chain map on $k[X_{\geqslant 2}]$. To do this we use a construction from \cite[Construction 2.3]{Briggs:2022}. This time on $k[X]$ we consider for each $i$ the derivation $\frac{d}{de_i}$ for which $\frac{d}{de_i}(e_i)=1$, $\frac{d}{de_i}(e_j)=0$ when $i\neq j$, and $\frac{d}{de_i}(x)=0$ if $x\in X_{\geqslant 2}$, and we take the linear combination \[\frac{d}{de_\alpha}=a_1 \frac{d}{de_1}+\cdots +a_n \frac{d}{de_n}\,.\]
The derivation $\frac{d}{de_\alpha}$ may not be a chain map, and we take its boundary $[\partial,\frac{d}{de_\alpha}]=\partial\frac{d}{de_\alpha}+\frac{d}{de_\alpha}\partial$. For degree reasons $[\partial,\frac{d}{de_\alpha}](X_1)=0$, and it follows that $[\partial,\frac{d}{de_\alpha}]$ induces a derivation on $k[X]/(X_1)=k[X_{\geqslant 2}]$, that we denote $\theta_\alpha\colon k[X_{\geqslant 2}] \to k[X_{\geqslant 2}]$.

We can extend $\frac{d}{de_\alpha}$ to a derivation $\frac{d}{de_\alpha}'$ on $k[X]\langle {y_1, \ldots, y_n} \rangle$ by setting 
$$\frac{d}{de_\alpha}'(y^{(p)}_i) = 0$$ 
for each $i$ and $p$. With all of this notation in place, we consider the diagram
\[
\begin{tikzcd}
k[X]\langle {y_1, \ldots, y_n} \rangle \ar[d,"{[\partial,\frac{d}{de_\alpha}'] + \frac{d}{dy_\alpha}}"'] \ar[r,"{\simeq}"] & k[X_{\geqslant 2}] \ar[d,"{\theta_\alpha}"]\\
k[X]\langle {y_1, \ldots, y_n}\rangle \ar[r,"{\simeq}"] & k[X_{\geqslant 2}]\,.
\end{tikzcd}
\]

A direct computation, by checking on the generators, shows that the diagram is commutative. It follows that on $\H\!\big(k[X_{\geqslant 2}]\big)$ we have
\[
\chi_\alpha = \H\!\left({\textstyle \frac{d}{dy_\alpha}}\right) = \H\!\left({\textstyle [\partial,\frac{d}{de_\alpha}'] + \frac{d}{dy_\alpha}}\right) = \H(\theta_a)\,.
\]
Returning to the homotopy Lie algebra, the definition of $\pi^*(R)$ implies that, for $i\geqslant 3$, there is a canonical isomorphism
\[
\pi^i(R)\cong \big((X_{\geqslant 2})/(X_{\geqslant 2})^2\big)_{i-1}^\vee\,.
\]
At the same time, the derivation $\theta_\alpha$ induces a map
\[
\overline{\theta}_\alpha \colon \big((X_{\geqslant 2})/(X_{\geqslant 2})^2\big)_{i+1}\longrightarrow \big((X_{\geqslant 2})/(X_{\geqslant 2})^2\big)_{i-1}\,.
\]
And the statement of \cite[Proposition 2.4]{Briggs:2022} is that the following diagram is commutative
\[
\begin{tikzcd}
\pi^{i+2}(R) \ar[d,"\cong"] & \pi^i(R) \ar[d,"\cong"]  \ar[l,"-\ad(\alpha)"']\\
\big((X_{\geqslant 2})/(X_{\geqslant 2})^2\big)_{i+1}^\vee & \big((X_{\geqslant 2})/(X_{\geqslant 2})^2\big)_{i-1}^\vee  \ar[l,"\overline{\theta}_\alpha^\vee"']\,;
\end{tikzcd}
\]
compare also the proof of \cite[Proposition 4.2]{Avramov/Halperin:1987}. Therefore, if $\ad(\alpha)^p=0$ then as well $\overline{\theta}_\alpha^p=0$, or in other words ${\theta}_\alpha^p(X_{\geqslant 2}) \subseteq (X_{\geqslant 2})^2$. Iterating this, there is for any $m$ some number $q$ such that ${\theta}_\alpha^q(X_{\geqslant 2}) \subseteq (X_{\geqslant 2})^m$.

Now by \ref{c_LScategory} and \ref{c_fiveauthorbound} there is an $m$ such that the inclusion $(X_{\geqslant 2})^m \hookrightarrow (X_{\geqslant 2})$ is a nullhomotopic chain map. Taking the corresponding $q$, the chain map \[\theta_\alpha^q \colon  k[X_{\geqslant 2}] \to  k[X_{\geqslant 2}]\] 
factors through this inclusion, and is therefore also nullhomotopic. In particular, $\chi_\alpha^q=\H(\theta_\alpha^q)=0$ on $\H\!\big(k[X_{\geqslant 2}]\big)$, and this concludes the proof.
\end{proof}

We now  provide several applications of \cref{t_bound,th_radical_support}. The first is \cref{intro upper bound on dim rad} from the introduction. 

\begin{theorem}\label{upper bound on dim rad}
Let $M$ be a nonzero perfect complex over $R$. Then
\[\sum \ell\ell_R \left( \H_n(M) \right) \ges \level^k_R(M)\ges \codim \V_R(R)+1 \geqslant \rank_k\rad^2(\pi^*(R))+1.\]
\end{theorem}

\begin{proof}
The first inequality is standard; see \cite[Theorem~6.2]{Avramov/Buchweitz/Iyengar/Miller:2010}. For the remaining inequalities, applying 
$- \otimes_R K^R$ yields the inequality
\begin{equation}
\label{eq1}
    \level^k_R(M) \ges \level_{K^R}^k (K^M)\,;
\end{equation}
cf.\@ \cref{c_levelsgodown}. 
Also, from \cref{t_bound} we have the first inequality
\begin{equation}
\label{eq2}
    \level^k_{K^R}(K^M) \ges \varepsilon_2(R)-\dim \V_R(M)+1=\varepsilon_2(R)-\dim\V_R(R)+1\,;
\end{equation}
for the equality note that $M$ is a nonzero perfect complex and so from \cite[3.3.2]{Pollitz:2019} we have that $\V_R(M)=\V_R(R)$.
Now applying \cref{th_radical_support} and combining the inequalities from \eqref{eq1} and  \eqref{eq2} we obtain the desired result. \end{proof}

\begin{remark}\label{r_cf_rank}
\cref{upper bound on dim rad} recovers and improves the bound from the main theorem in \cite{Avramov/Buchweitz/Iyengar/Miller:2010} which replaces the right hand-side in the inequality above with the conormal free rank of $R$. The {\bf conormal free rank} of $R$,
   denoted $\text{cf-rank}(R)$, is the largest free rank of a conormal module $I/I^2$ corresponding to a minimal Cohen presentation $Q \twoheadrightarrow \widehat{R} \cong Q/I$. To see that \cref{upper bound on dim rad} recovers \cite[Theorem~3]{Avramov/Buchweitz/Iyengar/Miller:2010}, we only need to recall the fact that $\rank_k \rad^2(\pi^*(R))\geqslant \text{cf-rank}(R)$ from  \cite{Iyengar:2001}. Furthermore, whenever $\V_R(R)$ is not a linear space, the bound in \cref{upper bound on dim rad} is strictly stronger than the one in \cite[Theorem~3]{Avramov/Buchweitz/Iyengar/Miller:2010} 
   as 
   \[
   \codim \V_R(R)>\codim \mathrm{span}_k \V_R(R)\ges \text{cf-rank}(R)\,.
   \]
  Finally, as noted in \cite{Avramov/Buchweitz/Iyengar/Miller:2010}, this also recovers a theorem from homotopy theory \cite{Carlsson:1983,Allday/Puppe:1993}: \emph{For a prime $p>0$, if $E$ is an elementary abelian $p$-group of rank $r$ that acts freely and cellulary on a finite CW-complex $X$, then}
   \[
   \sum\ell\ell_{\mathbb{F}_pE} \left( \H_n(X;\mathbb{F}_p) \right) \ges r+1\,.
   \]
\end{remark}

Together with the characterization of complete intersection rings in terms of $\V_R(R)$ \cite{Pollitz:2019},
Theorem \ref{th_radical_support} gives a conceptual generalization of the following result of Avramov and Halperin~\cite[Theorem~C]{Avramov/Halperin:1987}, which originally was used to prove that a  conjecture of Quillen on the cotangent complex holds in characteristic zero \cite{Quillen:1970}. The same result is also an ingredient in the recent proof of Vasconcelos' Conjecture on the conormal module~\cite{Briggs:2022}, as well as a recent new proof of Quillen's Conjecture in all characteristics~\cite{Briggs/Iyengar}, that had originally been settled by Avramov~\cite{Avramov:1999}.

\begin{corollary} \label{cor_AH_rad}
The ring $R$ is complete intersection if and only if every element of $\pi^2(R)$ is radical.
\end{corollary}

\begin{proof}
If $R$ is complete intersection then $\pi^i(R)=0$ for $i\geqslant 3$, so every element of $\pi^2(R)$ is central, and in particular radical. If $\rad^2(\pi^*(R))=\pi^2(R)$ then by Theorem \ref{th_radical_support} we have $\V_R(R)=\{{\bm 0}\}$. By \cite[Theorem~3.3.2]{Pollitz:2019} this implies that $R$ is complete intersection.
\end{proof}

\begin{remark}
For accessibility, we have treated here the \emph{absolute} case, concerning the homotopy Lie algebra and cohomological supports defined over a local ring. The discussion generalizes almost verbatim to the \emph{relative} case, beginning with a surjective local map $\varphi\colon Q\to R$ (with $Q$ not necessarily regular). As long as $\varphi$ has finite projective dimension, the same proof shows that Theorem \ref{th_radical_support} holds for the relative homotopy Lie algebra $\pi^*(\varphi)$ (cf.~\cite{Briggs:2018}) and the relative support variety $\V_\varphi(R)$ (cf.~\cite{Avramov/Iyengar:2007,Pollitz:2021}). We note that Avramov and Halperin's result Corollary \ref{cor_AH_rad} was stated in this greater generality in \cite{Avramov/Halperin:1987}, and since \cite[Theorem~3.3.2]{Pollitz:2019}  applies in this more general setting we also recover the general statement in \cite{Avramov/Halperin:1987}. 
\end{remark}

\begin{remark}\label{remark embedded deformation}
 Given an embedded deformation $R=S/(f)$ (see \Cref{ex_codepth} for a definition), the element $f$ gives rise to a free summand of the conormal module of $R$, and also to a central element in $\pi^2(R)$ and a hyperplane containing $\V_R(R)$; cf.\@ \cite{Iyengar:2001} and \cite[Proposition~5.3.7]{Pollitz:2021}, respectively. In \cite{Briggs/Grifo/Pollitz:2022}, the ring $R$ is said to have {\bf spanning support} if $\V_R(R)$ generates the whole affine space, as a vector space. It follows that if $R$ has spanning support then $\mathrm{cf\text{-}rank}(R)=0$ and $R$ does not have an embedded deformation. Determining whether a given ring has an embedded deformation is in general a difficult question. 

Avramov \cite[Problem 4.3]{Avramov:1989a} asked if every central element in $\pi^2(R)$ arises from an embedded deformation of $R$. This has been answered affirmatively in some specific cases in \cite{Avramov:1989a} and \cite{Lofwall:1994}.  Iyengar had previously shown that  free summands of $I/I^2$ give rise to central elements of $\pi^2(R)$ \cite{Iyengar:2001}, and Johnson has recently shown that if $I$ is a licci ideal, then $R$ has an embedded deformation if and only if $I/I^2$ has a free summand \cite{Johnson:2022}. However, Dupont has shown that there is a  (non-standard) graded ring $R$ with a central element in $\pi^2(R)$ that does \emph{not} admit any embedded deformation compatible with the grading \cite{Dupont}; the ungraded analogue remains open.
\end{remark}

\begin{remark} 
\Cref{th_radical_support} says that each radical element in $\pi^2(R)$ defines a hyperplane containing $\V_R(R)$, and therefore to solve \cite[Problem 4.3]{Avramov:1989a} it is sufficient to show that every hyperplane containing $\V_R(R)$ gives rise to an embedded deformation of $R$; whether this holds was asked by the third author in \cite[Question~6.3.8]{Pollitz:2021}. Answering Avramov's problem using this strategy would also show that elements of $\rad^2(\pi^*(R))$ are in fact central. 

Moreover, from a computational and experimental aspect,  $\V_R(R)$ is easier to compute than $\rad^2(\pi^*(R))$. Indeed, when $R$ is not complete intersection $\pi^*(R)$ is an infinite dimensional graded $k$-space having exponential growth, and testing whether an element is radical is not realistic in examples. On the other hand, to calculate  $\V_R(R)$  one needs only a system of higher homotopies on a minimal resolution of $\widehat{R}$ over $Q$ \cite[Section~7]{Eisenbud:1980}, and this amounts to computing a finite number of matrices.
\end{remark}

\begin{remark}
 We expect the embedding in \Cref{th_radical_support} to be an isomorphism. That is, when $k$ is algebraically closed, we expect an equality
\begin{equation}\label{eq_rad_hyp}
\rank_k (\rad^2(\pi^*(R)) ) = \varepsilon_2(R) - \rank_k (\mathrm{span}_k \V_R(R))\,.
\end{equation}
Such a result would make the connection between Avramov's question \cite[Problem 4.3]{Avramov:1989a} and the third author's question \cite[Question~6.3.8]{Pollitz:2021} even stronger: Together, (\ref{eq_rad_hyp}) and \cite[Question~6.3.8]{Pollitz:2021} amount to the statement that $R$ has an embedded deformation if and only if  $\rad^2(\pi^*(R))$ is nonzero. This would strengthen \cite[Problem 4.3]{Avramov:1989a} from central elements to radical elements.
\end{remark}

\begin{remark}
Notice that in the proof of \Cref{upper bound on dim rad} the following inequality was established:
\[\level_{K^R}^k (K^M) \geqslant \rank_k (\rad^2(\pi^*(R)))\,.\]
Combining this with \Cref{l_levelvscat} yields
\[\cat(R) \geqslant \rank_k\rad^2(\pi^*(R))\,.\]
\end{remark}

Finally, we note that the methods above allow us to recover, by a new argument, one of the results of the famous \emph{Five Author Paper} \cite[Theorem A]{Felix/Halperin/Jacobsson/Lofwall/Thomas:1988}:

\begin{corollary}
For any local ring $R$,
\[
\codepth(R)\geqslant \rank_k (\rad^2(\pi^*(R)))\,.
\]
\end{corollary}

\begin{proof}
Apply \Cref{upper bound on dim rad} when $M$ is the Koszul complex on a minimal generating set for the maximal ideal of $R$; the inequality follows from the fact that in that case $\level^k_R(M) = \codepth(R)$. This is simply a consequence of the rigidity of the Koszul complex \cite[Example~6.7]{Avramov/Buchweitz/Iyengar/Miller:2010}.
\end{proof}

\section{The realizability problem for Golod rings}\label{s_golod}
For a local ring $R$ with residue field $k$, there is the following coefficient-wise inequality on the Poincaré series $\mathrm{P}_k^R(t)$ of $k$  due to Serre: 
\[
\mathrm{P}_k^R(t)\preccurlyeq\frac{(1+t)^{\varepsilon_1(R)}}{1-t\sum_{i=1}^{\codepth R}\rank_k\H_i(K^R) t^i}\,;
\]
see for example \cite[Proposition~4.1.4]{Avramov:2010}. We say $R$ is \textbf{Golod} if equality holds. By definition these rings possess extremal homological behavior; see \cite{Golod} or \cite[Section~5]{Avramov:2010}. 
There are two distinct cases: first, when the codepth of $R$ is at most one (and hence $R$ is a hypersurface ring), then $R$ is always Golod and has easy to  understand asymptotic properties; second, when $R$ is Golod with codepth at least two, resolutions of modules typically have  exponential growth. 

In this section we show that the extremal behavior of Golod rings with 
codepth at least two is also reflected in the limitations on the cohomological supports that are realizable, that in turn, also puts restrictions on the lattice of thick subcategories over such rings.

\begin{theorem}\label{golod}
    If a local ring $R$ is Golod with $\codepth R\geqslant 2$, then $R$ has full support $\V_R(R)=\spec \S$. Moreover, the cohomological support of any nonzero object $M$ of $\Df(R)$ satisfies the inequality $\codim \V_R(M)\les 1$, and every closed cone having codimension at most one is the cohomological support variety of some object in $\Df(R)$.
\end{theorem}

\begin{proof}
    By \cite[Proposition 9]{Briggs:2018}, when $R$ is Golod we have $\cat(R)=1$. Hence, by \Cref{t_bound} we conclude that $\codim \V_R(M) \les 1$ whenever $M$ is a nonzero object in $\Df(R)$. The realizability assertion follows from \cref{ex_koszul}.
    
 So it remains to show $\V_R(R)=\spec \S.$ To this end, using the notation from \cref{t functor}, since $R$ is Golod 
    \[
    \t (K^R) \simeq k \ltimes V\,
    \]
 as a $\Lambda$-module,  where $V=\{V_i\}_{i\ges 1}$ is a graded $k$-vector space and where the $\Lambda$-action on  $k\ltimes V$ is trivial: 
 \[
 \Lambda_1\cdot k= V_1\quad\text{and}\quad \Lambda_1 \cdot V=0\,;
 \] 
 see \cite[Theorem~4.6]{Golod}. Since $R$ is not a hypersurface ring  $V_2\neq 0$ and so $k$ is a summand of $k\ltimes V$ over $\Lambda$. Therefore, 
     \[
     \supp_\S(\h\t K^R)=\supp_\S(\h k)\,.
     \]
     Finally, since $\h k\simeq \S$ the desired results follow from \cref{p_koszulbggsupport}. 
\end{proof}

\begin{remark}
The fact that $\V_R(R)=\spec \S$ when $R$ is a  non-hypersurface Golod ring can also be established using the structure of a system of higher homotopies for $\widehat{R}$ over $Q$. This makes use of the A$_\infty$-structure on a minimal $Q$-free resolution of $\widehat{R}$ and will be explained in future work. 
\end{remark}

\begin{remark}\label{remark Golod no embedded defs}
As described in \Cref{remark embedded deformation}, \Cref{golod} implies that if $R$ is a non-hypersurface Golod ring, then the conormal free rank of $R$ is zero, and in particular $R$ does not have an embedded deformation. Notice that the fact that $R$ has no embedded deformation also follows from the fact that $\pi^{\geqslant 2}(R)$ is free when $R$ is Golod, using \cite[Theorem 4.6]{Golod}. 
\end{remark}

Our results show that the bounded derived category of a Golod local ring looks markedly different to that of a complete intersection ring; compare with \cref{rem_ci_thick_subcats} and \cref{cor_thick_intersection}, as well as \cite[Theorems A \& B]{Elagin/Lunts:2021}.

\begin{corollary}\label{cor Lzeta intersection golod}
\label{proxysmall}
 Let $R$ be a Golod local ring with $\codepth R\geqslant 2$. 
 \begin{enumerate}
     \item \label{proxysmall1}For each $\zeta \in \S$ of positive degree, $L_\zeta$ is not proxy small. 
     \item \label{proxysmall2} If $\zeta,\zeta'$ are in $\S^{>0}$ with $\gcd(\zeta,\zeta')=1$, then  
\[
\thick_R L_\zeta \cap \thick_R L_{\zeta'}=0\,.
\]
 \end{enumerate}
\end{corollary}

\begin{proof}
    By \cite[3.3.2]{Pollitz:2019}, if $L_\zeta$ is proxy small then \[ \V_R(R)\subseteq \V_R(L_\zeta)=\cV(\zeta);\]
    the second equality uses \cref{ex_koszul}. However this contradicts \Cref{golod}, and justifies \cref{proxysmall1}. 
    
    For \cref{proxysmall2}, if $M$ is in $\thick_R L_\zeta \cap \thick_R L_{\zeta'}$, then 
    \[\V_R(M)\subseteq \V_R(\zeta)\cap \V_R(\zeta')=\cV(\zeta)\cap \cV(\zeta')= \cV(\zeta,\zeta')\,;\]
    the first containment is again \cite[3.3.2]{Pollitz:2019}, while the next equality follows from \cref{ex_koszul}. On the other hand, since $\zeta, \zeta'$ form a regular sequence, $\cV(\zeta,\zeta')$ must have codimension exactly $2$. Therefore, $\V_R(M)$ would also have codimension at least $2$, which contradicts the first part of the theorem.
\end{proof}

We end with an example of a class of Golod rings where one can explicitly realize all codimension one cones using modules, generalizing \cite[Example 6.4.3]{Pollitz:2021}.

\begin{example}
Let $R$ be a codepth two non-complete intersection ring, and assume $R$ that is complete and $k$ is algebraically closed.  In this case, $R\cong Q/I$ where  $(Q,\m)$ is a regular local ring, and $I=(xy,yz)$,  with $x,z$ a regular sequence  and $y$ a nonzero element of $\m$. 

For each $(p, q) \in \mathbb{A}^2_k \smallsetminus \{ (0,0) \}$ we consider the $R$-module
\[
M_{(p,q)} \coloneqq R/(px+qz).
\]
If $p \neq 0$ then $M_{(p,q)} \cong Q  / (px+qz,yz)$, and if $p = 0$ then $M_{(0,b)} \cong Q / (xy,z)$. In either case $M_{(p,q)}$ is a complete intersection ring, so setting $J_{(p,q)} = (xy,yz,px+qy)$, by \cref{ex_bej}, we obtain
\[
\V_R(M_{(p,q)}) = \ker \left( I/\m I \to J_{(p,q)} / \m J_{(p,q)} \right) = \{ (a,b) \in \mathbb{A}_k^2 \mid qa = pb  \}.
\]
Note that any codimension one cone in $\mathbb{A}^2_k$ is a finite union of lines through the origin. Given such a variety
\[
V= \bigcup_{i=1}^n\{ (a,b) \in \mathbb{A}_k^2 \mid q_ia = pb_i  \}\,,
\]
we obtain
$
\V_R(M)=V
$ by setting $M=\bigoplus_{i=1}^n M_{(p_i,q_i)}$ and  making use of  \cite[Proposition~5.1.2]{Pollitz:2021}.
\end{example}

\section*{Acknowledgements}

We thank Srikanth Iyengar for helpful comments on a previous version of the paper. We are also very grateful to the referee for their excellent suggestions and comments that greatly improved this work.
Briggs was partly funded by NSF grant DMS-1928930, and partly funded by the European Union under the Grant Agreement no.\ 101064551, Hochschild. Grifo was supported by NSF grant DMS-2140355 and NSF CAREER grant DMS-2236983. Pollitz was supported by NSF grants DMS-1840190, DMS-2002173, and DMS-2302567.


\bibliographystyle{alpha}
\bibliography{references}

\newcommand{\etalchar}[1]{$^{#1}$}
\begin{thebibliography}{AINSW19}

\bibitem[AB00a]{Avramov/Buchweitz:2000a}
Luchezar~L. Avramov and Ragnar-Olaf Buchweitz.
\newblock Homological algebra modulo a regular sequence with special attention
  to codimension two.
\newblock {\em J. Algebra}, 230(1):24--67, 2000.

\bibitem[AB00b]{Avramov/Buchweitz:2000b}
Luchezar~L. Avramov and Ragnar-Olaf Buchweitz.
\newblock Support varieties and cohomology over complete intersections.
\newblock {\em Invent. Math.}, 142(2):285--318, 2000.

\bibitem[ABIM10]{Avramov/Buchweitz/Iyengar/Miller:2010}
Luchezar~L. Avramov, Ragnar-Olaf Buchweitz, Srikanth~B. Iyengar, and Claudia
  Miller.
\newblock Homology of perfect complexes.
\newblock {\em Adv. Math.}, 223(5):1731--1781, 2010.

\bibitem[AH87]{Avramov/Halperin:1987}
Luchezar~L. Avramov and Stephen Halperin.
\newblock On the nonvanishing of cotangent cohomology.
\newblock {\em Comment. Math. Helv.}, 62(2):169--184, 1987.

\bibitem[AI07]{Avramov/Iyengar:2007}
Luchezar~L. Avramov and Srikanth~B. Iyengar.
\newblock Constructing modules with prescribed cohomological support.
\newblock {\em Illinois J. Math.}, 51(1):1--20, 2007.

\bibitem[AI10]{Avramov/Iyengar:2010}
Luchezar~L. Avramov and Srikanth~B. Iyengar.
\newblock Cohomology over complete intersections via exterior algebras.
\newblock In {\em Triangulated categories}, volume 375 of {\em London Math.
  Soc. Lecture Note Ser.}, pages 52--75. Cambridge Univ. Press, Cambridge,
  2010.

\bibitem[AI18]{Avramov/Iyengar:2018}
Luchezar~L. Avramov and Srikanth~B. Iyengar.
\newblock Restricting homology to hypersurfaces.
\newblock In {\em Geometric and topological aspects of the representation
  theory of finite groups}, volume 242 of {\em Springer Proc. Math. Stat.},
  pages 1--23. Springer, Cham, 2018.

\bibitem[AINSW19]{Avramov/Iyengar/Nasseh/SatherWagstaff:2019}
Luchezar~L. Avramov, Srikanth~B. Iyengar, Saeed Nasseh, and Sean
  Sather-Wagstaff.
\newblock Homology over trivial extensions of commutative {DG} algebras.
\newblock {\em Comm. Algebra}, 47(6):2341--2356, 2019.

\bibitem[And18]{Andre:2018}
Yves Andr\'{e}.
\newblock La conjecture du facteur direct.
\newblock {\em Publ. Math. Inst. Hautes \'{E}tudes Sci.}, 127:71--93, 2018.

\bibitem[AP93]{Allday/Puppe:1993}
C.~Allday and V.~Puppe.
\newblock {\em Cohomological methods in transformation groups}, volume~32 of
  {\em Cambridge Studies in Advanced Mathematics}.
\newblock Cambridge University Press, Cambridge, 1993.

\bibitem[Avr75]{Avramov:1975}
Luchezar~L. Avramov.
\newblock Flat morphisms of complete intersections.
\newblock {\em Dokl. Akad. Nauk SSSR}, 225(1):11--14, 1975.

\bibitem[Avr77]{Avramov:1977}
Luchezar~L. Avramov.
\newblock Homology of local flat extensions and complete intersection defects.
\newblock {\em Math. Ann.}, 228(1):27--37, 1977.

\bibitem[Avr86]{Golod}
Luchezar~L. Avramov.
\newblock Golod homomorphisms.
\newblock In {\em Algebra, algebraic topology and their interactions
  ({S}tockholm, 1983)}, volume 1183 of {\em Lecture Notes in Math.}, pages
  59--78. Springer, Berlin, 1986.

\bibitem[Avr89a]{Avramov:1989a}
Luchezar~L. Avramov.
\newblock Homological asymptotics of modules over local rings.
\newblock In {\em Commutative algebra ({B}erkeley, {CA}, 1987)}, volume~15 of
  {\em Math. Sci. Res. Inst. Publ.}, pages 33--62. Springer, New York, 1989.

\bibitem[Avr89b]{Avramov:1989}
Luchezar~L. Avramov.
\newblock Modules of finite virtual projective dimension.
\newblock {\em Invent. Math.}, 96(1):71--101, 1989.

\bibitem[Avr99]{Avramov:1999}
Luchezar~L. Avramov.
\newblock Locally complete intersection homomorphisms and a conjecture of
  {Q}uillen on the vanishing of cotangent homology.
\newblock {\em Ann. of Math. (2)}, 150(2):455--487, 1999.

\bibitem[Avr10]{Avramov:2010}
Luchezar~L. Avramov.
\newblock Infinite free resolutions.
\newblock In {\em Six lectures on commutative algebra}, Mod. Birkh{\"a}user
  Class., pages 1--118. Birkh{\"a}user Verlag, Basel, 2010.

\bibitem[Ben17]{Benson:2017}
David~J. Benson.
\newblock {\em Representations of elementary abelian {$p$}-groups and vector
  bundles}, volume 208 of {\em Cambridge Tracts in Mathematics}.
\newblock Cambridge University Press, Cambridge, 2017.

\bibitem[Ber07]{Bergh}
Petter~Andreas Bergh.
\newblock On support varieties for modules over complete intersections.
\newblock {\em Proc. Amer. Math. Soc.}, 135(12):3795--3803, 2007.

\bibitem[BGG78]{Bernstein/Gelfand/Gelfand:1978}
I.~N. Bern\v{s}te\u{\i}n, I.~M. Gel'fand, and S.~I. Gel'fand.
\newblock Algebraic vector bundles on {${\bf P}^{n}$} and problems of linear
  algebra.
\newblock {\em Funktsional. Anal. i Prilozhen.}, 12(3):66--67, 1978.

\bibitem[BGP22]{Briggs/Grifo/Pollitz:2022}
Benjamin Briggs, Elo\'{\i}sa Grifo, and Josh Pollitz.
\newblock Constructing nonproxy small test modules for the complete
  intersection property.
\newblock {\em Nagoya Math. J.}, 246:412--429, 2022.

\bibitem[BH86]{Bogvad/Halperin/1986}
Rikard B{\o}gvad and Stephen Halperin.
\newblock On a conjecture of {R}oos.
\newblock In {\em Algebra, algebraic topology and their interactions
  ({S}tockholm, 1983)}, volume 1183 of {\em Lecture Notes in Math.}, pages
  120--127. Springer, Berlin, 1986.

\bibitem[BH98]{Bruns/Herzog:1998}
Winfried Bruns and J{\"u}rgen Herzog.
\newblock {\em {C}ohen-{M}acaulay rings}, volume~39 of {\em Cambridge Studies
  in Advanced Mathematics}.
\newblock Cambridge University Press, Cambridge, revised edition, 1998.

\bibitem[Bha18]{Bhatt:2018}
Bhargav Bhatt.
\newblock On the direct summand conjecture and its derived variant.
\newblock {\em Invent. Math.}, 212(2):297--317, 2018.

\bibitem[BI23]{Briggs/Iyengar}
Benjamin Briggs and Srikanth Iyengar.
\newblock Rigidity properties of the cotangent complex.
\newblock {\em J. Amer. Math. Soc.}, 36(1):291--310, 2023.

\bibitem[BIK08]{Benson/Iyengar/Krause:2008}
Dave Benson, Srikanth~B. Iyengar, and Henning Krause.
\newblock Local cohomology and support for triangulated categories.
\newblock {\em Ann. Sci. {\'E}c. Norm. Sup{\'e}r. (4)}, 41(4):573--619, 2008.

\bibitem[Bri18]{Briggs:2018}
Benjamin Briggs.
\newblock {\em Local Commutative Algebra and {H}ochschild Cohomology Through
  the Lens of {K}oszul Duality}.
\newblock PhD thesis, University of Toronto, (2018).
\newblock \url{https://www.math.utah.edu/~briggs/briggsthesis.pdf}.

\bibitem[Bri22]{Briggs:2022}
Benjamin Briggs.
\newblock Vasconcelos' conjecture on the conormal module.
\newblock {\em Invent. Math.}, 227(1):415--428, 2022.

\bibitem[BvdB03]{Bondal/VanDenBergh:2003}
Alexey~I. Bondal and Michel van~den Bergh.
\newblock Generators and representability of functors in commutative and
  noncommutative geometry.
\newblock {\em Mosc. Math. J.}, 3(1):1--36, 258, 2003.

\bibitem[BW15]{Burke/Walker:2015}
Jesse Burke and Mark~E. Walker.
\newblock Matrix factorizations in higher codimension.
\newblock {\em Trans. Amer. Math. Soc.}, 367(5):3323--3370, 2015.

\bibitem[Car83]{Carlsson:1983}
G.~Carlsson.
\newblock On the homology of finite free {$({\bf Z}/2)^{n}$}-complexes.
\newblock {\em Invent. Math.}, 74(1):139--147, 1983.

\bibitem[Car84]{Carlson:1984}
Jon~F. Carlson.
\newblock The variety of an indecomposable module is connected.
\newblock {\em Invent. Math.}, 77(2):291--299, 1984.

\bibitem[DGI06]{Dwyer/Greenlees/Iyengar:2006b}
William~G. Dwyer, John P.~C. Greenlees, and Srikanth~B. Iyengar.
\newblock Finiteness in derived categories of local rings.
\newblock {\em Comment. Math. Helv.}, 81(2):383--432, 2006.

\bibitem[Dup97]{Dupont}
Nicolas Dupont.
\newblock Un diviseur de z\'{e}ro induisant un \'{e}l\'{e}ment d'homotopie
  central.
\newblock {\em Bull. Soc. Math. France}, 125(3):337--344, 1997.

\bibitem[EHT{\etalchar{+}}04]{Erdmann:2004}
Karin Erdmann, Miles Holloway, Rachel Taillefer, Nicole Snashall, and {\O}yvind
  Solberg.
\newblock Support varieties for selfinjective algebras.
\newblock {\em $K$-Theory}, 33(1):67--87, 2004.

\bibitem[Eis80]{Eisenbud:1980}
David Eisenbud.
\newblock Homological algebra on a complete intersection, with an application
  to group representations.
\newblock {\em Trans. Amer. Math. Soc.}, 260(1):35--64, 1980.

\bibitem[EL22]{Elagin/Lunts:2021}
Alexey Elagin and Valery~A. Lunts.
\newblock Derived categories of coherent sheaves on some zero-dimensional
  schemes.
\newblock {\em J. Pure Appl. Algebra}, 226(6):Paper No. 106939, 30, 2022.

\bibitem[FH80]{Felix/Halperin:1980}
Yves F\'elix and Stephen Halperin.
\newblock Rational {L}.-{S}. category and its applications.
\newblock {\em Publ. U.E.R. Math. Pures Appl. IRMA}, 2(3):exp. no. 5, 84, 1980.

\bibitem[FHJ{\etalchar{+}}88]{Felix/Halperin/Jacobsson/Lofwall/Thomas:1988}
Yves F\'{e}lix, Stephen Halperin, Carl Jacobsson, Clas L\"{o}fwall, and
  Jean-Claude Thomas.
\newblock The radical of the homotopy {L}ie algebra.
\newblock {\em Amer. J. Math.}, 110(2):301--322, 1988.

\bibitem[GL69]{Gulliksen/Levin:1969}
Tor~H. Gulliksen and Gerson Levin.
\newblock {\em Homology of local rings}.
\newblock Queen's Paper in Pure and Applied Mathematics, No. 20. Queen's
  University, Kingston, Ont., 1969.

\bibitem[IPS22]{Iyengar/Pollitz/Sanders:2022}
Srikanth~B. Iyengar, Josh Pollitz, and William~T. Sanders.
\newblock Cohomological supports of tensor products of modules over commutative
  rings.
\newblock {\em Res. Math. Sci.}, 9(2):Paper No. 25, 15, 2022.

\bibitem[Iye01]{Iyengar:2001}
Srikanth Iyengar.
\newblock Free summands of conormal modules and central elements in homotopy
  {L}ie algebras of local rings.
\newblock {\em Proc. Amer. Math. Soc.}, 129(6):1563--1572, 2001.

\bibitem[Joh22]{Johnson:2022}
Mark~R. Johnson.
\newblock Splitting the conormal module for licci ideals.
\newblock {\em J. Commut. Algebra}, 14(1):55--60, 2022.

\bibitem[Jor02]{Jorgensen:2002}
David~A. Jorgensen.
\newblock Support sets of pairs of modules.
\newblock {\em Pacific J. Math.}, 207(2):393--409, 2002.

\bibitem[KK65]{KiehlKunz}
Reinhardt Kiehl and Ernst Kunz.
\newblock Vollst\"{a}ndige {D}urchschnitte und {$p$}-{B}asen.
\newblock {\em Arch. Math.}, 16:348--362, 1965.

\bibitem[L{\"{o}}f94]{Lofwall:1994}
Clas L{\"{o}}fwall.
\newblock Central elements and deformations of local rings.
\newblock {\em J. Pure Appl. Algebra}, 91(1-3):183--192, 1994.

\bibitem[LP21]{Liu/Pollitz:2021}
Jian Liu and Josh Pollitz.
\newblock Duality and symmetry of complexity over complete intersections via
  exterior homology.
\newblock {\em Proc. Amer. Math. Soc.}, 149(2):619--631, 2021.

\bibitem[Pol19]{Pollitz:2019}
Josh Pollitz.
\newblock The derived category of a locally complete intersection ring.
\newblock {\em Adv. Math.}, 354:106752, 18, 2019.

\bibitem[Pol21]{Pollitz:2021}
Josh Pollitz.
\newblock Cohomological supports over derived complete intersections and local
  rings.
\newblock {\em Math. Z.}, 299(3-4):2063--2101, 2021.

\bibitem[Qui70]{Quillen:1970}
Daniel Quillen.
\newblock On the (co-) homology of commutative rings.
\newblock In {\em Applications of {C}ategorical {A}lgebra ({P}roc. {S}ympos.
  {P}ure {M}ath., {V}ol. {XVII}, {N}ew {Y}ork, 1968)}, pages 65--87. Amer.
  Math. Soc., Providence, R.I., 1970.

\bibitem[Qui71]{Quillen:1971}
Daniel Quillen.
\newblock The spectrum of an equivariant cohomology ring. {I}, {II}.
\newblock {\em Ann. of Math. (2)}, 94:549--572; ibid. (2) 94 (1971), 573--602,
  1971.

\bibitem[Rou08]{Rouquier:2008}
Rapha\"{e}l Rouquier.
\newblock Dimensions of triangulated categories.
\newblock {\em J. K-Theory}, 1(2):193--256, 2008.

\bibitem[SFB97]{Suslin/Friedlander/Bendel:1997}
Andrei Suslin, Eric~M. Friedlander, and Christopher~P. Bendel.
\newblock Support varieties for infinitesimal group schemes.
\newblock {\em J. Amer. Math. Soc.}, 10(3):729--759, 1997.

\bibitem[Ste14]{Stevenson:2014a}
Greg Stevenson.
\newblock Subcategories of singularity categories via tensor actions.
\newblock {\em Compositio Mathematica}, 150(2):229–272, 2014.

\bibitem[Tat57]{Tate:1957}
John Tate.
\newblock Homology of {N}oetherian rings and local rings.
\newblock {\em Illinois J. Math.}, 1:14--27, 1957.

\end{thebibliography}

\end{document}